\documentclass[lettersize,a4paper]{elsarticle}
\usepackage{amsmath}
\usepackage{amssymb}
\usepackage{amsthm}
\usepackage{amsfonts}
\usepackage{geometry}
\usepackage{hyperref}
\usepackage{mathdots}
\hypersetup{
    colorlinks=true,
    linkcolor=blue,
    filecolor=magenta,      
    urlcolor=cyan,
    }
\usepackage[normalem]{ulem}
\usepackage[nothing]{algorithm}
\makeatletter
\renewcommand{\ALG@name}{Process} 
\makeatother
\usepackage{algpseudocode}
\usepackage{tikz}
\usetikzlibrary{shapes, arrows, calc, arrows.meta, fit, positioning} 
\usepackage{float}

\algrenewcommand\algorithmicrequire{\textbf{Input:}}
\algrenewcommand\algorithmicensure{\textbf{Output:}}

\makeatletter
\def\ps@pprintTitle{%
   \let\@oddhead\@empty
   \let\@evenhead\@empty
   \let\@oddfoot\@empty
   \let\@evenfoot\@oddfoot
}
\makeatother

\theoremstyle{plain}
\newtheorem{theorem}{Theorem}
\newtheorem{proposition}[theorem]{Proposition}
\newtheorem{lemma}[theorem]{Lemma}
\newtheorem{corollary}[theorem]{Corollary}

\newtheorem{definition}[theorem]{Definition}
\newtheorem{example}[theorem]{Example}

\theoremstyle{remark}
\newtheorem{remark}[theorem]{Remark}
\newcommand{\R}{\mathbb{R}}

\newcommand{\diam}{\mathrm{diam}}
\newcommand{\F}{\mathbb F} 
\newcommand{\dist}{\mathrm{dist}}
\newcommand{\ad}{\mathrm{ad}}

\let\originalleft\left
\let\originalright\right
\renewcommand{\left}{\mathopen{}\mathclose\bgroup\originalleft}
\renewcommand{\right}{\aftergroup\egroup\originalright} 
\begin{document}

\begin{frontmatter}

\title{Diameter of Commuting Graphs of Lie Algebras}

\author[VNU1,VNU2]{Hieu V. Ha}
\ead{hieuhv@uel.edu.vn}

\author[HSU]{Hoa D. Quang \corref{mycorrespondingauthor}}
\cortext[mycorrespondingauthor]{Corresponding author}
\ead{hoa.duongquang@ufm.edu.vn}

\author[VNU1,VNU2]{Vu A. Le}
\ead{vula@uel.edu.vn}

\author[DTU]{Tuyen T. M. Nguyen}
\ead{ntmtuyen@dthu.edu.vn}

\address[VNU1]{
University of Economics and Law, Ho Chi Minh City, Vietnam}

\address[VNU2]{Vietnam National University, Ho Chi Minh City, Vietnam}

\address[HSU]{University of Finance - Marketing, Ho Chi Minh City, Vietnam}

\address[DTU]{Dong Thap University, Dong Thap Province, Vietnam}

\begin{abstract}
    In this paper, we study the connectedness of the commuting graph of a general Lie algebra and provide a process to determine whether the commuting graph is connected or not, as well as to compute an upper bound for its diameter. In addition, we will examine the connectedness and diameter of the commuting graphs of some remarkable classes of Lie algebras, including: (1) a class of Lie algebras with one- or two-dimensional derived algebras; and (2) a class of solvable Lie algebras over the real field of dimension up to $4$.
\end{abstract}

\end{frontmatter}

\section{Introduction}
The connections between graph theory and other mathematical structures, such as group theory and ring theory, have become an important problem in recent years. There are many papers on investing in these connections (for example, see \cite{Anderson2008, Anderson1999, Akbari2004, Akbari2008, Akbari2006}). Let $\cal R$ be a noncommutative ring, and let $Z(\cal R)$ denote the center $\cal R$, that is, $Z(\cal R)$ is the set of elements $x\in \cal R$ such that $xy=yx$ for every $y\in \cal R$. Then the commuting graph of $\cal R$ is defined as a simple undirected graph, denoted by $\Gamma(\cal R)$, which has the vertex set $V=\mathcal R \setminus Z(\mathcal R)$ in which two vertices $x,\, y \in V$ are adjacent if and only if $xy=yx$ \cite{Akbari2004}. For this graph, its diameter, denoted by $\diam(\Gamma(\mathcal R))$, is defined as the supremum of the distances of any two distinct vertices $u$ and $v$ in $V$. 
In \cite{Akbari2004, Akbari2008, Akbari2006, Dorbidi2024, Miguel2016, Nam2024}, the connectedness and diameter of the commuting graphs of some important subsets of the ring $M_n(\mathbb F)$ of all $n\times n$ matrices were studied. 
Note that the matrix algebra $M_n(\mathbb{F})$ is also a Lie algebra in which the Lie bracket is defined by $[A,B]=AB-BA$. In \cite{Wang2017}, Wang and Xia extended these terms to Lie algebras and obtained the diameters of the commuting graphs of simple, finite-dimensional Lie algebras over an algebraically closed field of characteristic 0. More precisely, let $\mathcal L$ be such a Lie algebra of rank $l$. If $l=1$ then the commuting graph of $\mathcal L$ is disconnected. Otherwise, $\Gamma(\mathcal L)$ is connected and its diameter is bounded below and above by 3 and 6, respectively.

In this paper, we extend these terms to more general classes of Lie algebras. In particular, we prove that the diameter of the commuting graph of any direct sum of two non-commutative Lie algebras is upper bounded by $3$ and give an example to prove that this bound is sharp (Section~\ref{section:general}). In this section, we also give some remarkable conditions a Lie algebra to have a connected commuting graph. Using these conditions, we obtain the connectedness of the commuting graphs of all Lie algebras having one- or two-dimensional derived algebras (Section~\ref{sec:one-two-dimension}). We also propose a new method, which is written as a process, to obtain the connectedness and an upper bound for the diameter of the commuting graph of a general Lie algebra (Section~\ref{sec:algorithm}). Applying this method, we determine the connectedness and diameter of all real solvable Lie algebras of dimension up to $4$ (Section~\ref{sec:<=5}).

\section{Preliminaries}\label{sec:preliminaries}
In this paper, a graph $\Gamma = (V,E)$ is always a simple undirected graph with vertex set $V$ and edge set $E$. For two distinct vertices $u, v \in V$, we will denote the edge joined $u$ and $v$ by $\{u,v\}$. 
If such an edge exists, we say that $u$ and $v$ are {\it adjacent}.
Otherwise, two vertices $u$ and $v$ are said to be non-adjacent. 
A {\it path} (joining $v_1$ and $v_k$) is a sequence of pairwise distinct vertices $v_1, v_2, \ldots, v_k$ such that $\{v_i,v_{i+1}\} \in E$ for every $i=1, 2, \ldots, k-1$. This path will be written as $\{v_1,v_2, \ldots,v_k\}$. The {\it length} of a path $P$ is the number of edges in $P$. The {\it distance} between two vertices $u$ and $v$, denoted $\dist(u,v)$, is the length of the shortest path joining them.  The graph $\Gamma$ is said to be {\it connected} if every pair of distinct vertices is joined by a path. Otherwise, it is disconnected.


\begin{definition}\cite{Chartrand2004}
    The diameter of a connected graph $\Gamma$, denoted $\diam(\Gamma)$, is defined by
    \begin{equation}
        \diam(\Gamma) = \sup\{\dist(u,v) \,|\, u\neq v\}.
    \end{equation}
\end{definition}
\noindent It can be seen directly from the above definition that $\diam(\Gamma)=1$ if and only if $\Gamma$ is a complete graph, i.e., every pair of distinct vertices is adjacent. Moreover, it should be noted that $\diam(\Gamma) < \infty$ if $\Gamma$ is a finite graph. Otherwise, if $\Gamma$ is an infinite graph, its diameter is not necessary to be less than infinity.

    Let $\mathcal L$ be a Lie algebra over a field $\mathbb F$. Then the {\it center} of $\cal L$, denoted by $Z(\cal L)$, is defined as the set of all elements $x\in \cal L$ such that $[x,\mathcal{L}]=0$, i.e., 
\begin{equation*}
    Z(\mathcal{L}):=\{x \in \mathcal{L} \, | \, [x,y]=0 \text{ for every } y \in \cal L \}.
\end{equation*}
It is clear that
$Z(\cal L)$ is an ideal of $\cal L$. 
By definition, the direct sum of two Lie algebras ${\cal L}_1$, ${\cal L}_2$, denoted by ${\cal L}_1 \oplus {\cal L}_2$, is the direct sum of two vector spaces ${\cal L}_1$ and ${\cal L}_2$ equipped the following Lie bracket:
\begin{equation}
    [x_1 + x_2, y_1 + y_2] := [x_1, y_1]  +  [x_2, y_2], \, \forall x_1, y_1 \in {\cal L}_1; \, \forall x_2, y_2 \in {\cal L}_2.
\end{equation}
A Lie algebra $\cal L$ is said to be {\it decomposable} if $\mathcal L$ can be decomposed into a direct sum of two non-trivial Lie subalgebras. Otherwise, $\mathcal L$ is said to be {\it indecomposable}. 

Besides, it is well-known that a Lie algebra of dimension $n$ can be defined via its structure constants. More precisely, a Lie algebra of dimension $n$ over a field $\mathbb F$, spanned by $\{e_1, e_2, \ldots, e_n\}$, can be defined directly by structure constants $\{(c_{ij})^k \mid 1 \leq i < j \leq n, 1 \leq k \leq n\}$ as follows:
\begin{equation}
    [e_i,e_j] = c_{ij}^1 e_1 + c_{ij}^2 e_2 + \cdots + c_{ij}^n e_n.
\end{equation}
Of course, the structure constants are not completely arbitrary. They must satisfy a relationship that is directly deduced from the Jacobi identity for all the triples $(i, j, k),\, 1 \leq i < j < k \leq n$ \cite{Kirillov1976}. 
Therefore, 
in this paper, we will use the following notation
\begin{equation}
    \mathcal L = \left< e_1, e_2, \ldots, e_n\right> \colon [e_i,e_j] = c_{ij}^1 e_1 + c_{ij}^2e_2+\cdots +c_{ij}^ne_n;\, 
     (i,j) \in M 
\end{equation}
to indicate a Lie algebra of dimension $n$ spanned by $\{e_1, \ldots, e_n\}$ and having their non-trivial Lie brackets: 
\begin{equation}
     [e_i,e_j] = c_{ij}^1 e_1 + c_{ij}^2e_2+\cdots +c_{ij}^ne_n; \, \forall 
     (i,j) \in M. 
\end{equation}
Here, $M$ is some nonempty set of $\{1, \cdots , n \} \times \{1, \cdots , n \}$ such that if $(i, j) \in M$ then $i < j$.
We refer the reader to \cite{Chartrand2004} for more details on graph theory and to \cite{Snob} for more details about Lie algebras.



We next introduce a basic concept of this paper, that is, the {\it commuting graph of a noncommutative Lie algebra}.

\begin{definition} \cite{Wang2017}
    Let $\cal L$ be a noncommutative Lie algebra over a field $\mathbb{F}$ with center $Z(\cal L)$. Then the commuting graph $\Gamma(\cal L)$ $= (V_{\cal L}, E_{\cal L})$ of $\cal L$ is the simple undirected graph which have
    \begin{enumerate}
        \item[i)] the vertex set $V_{\cal L}:= \mathcal{L}\setminus Z(\cal L)$, and
        \item[ii)] the edge set $E_{\cal L}:= \left\{ \{x, y\} \,|\, [x,y]=0 \right\}$.
    \end{enumerate}
\end{definition}
If $\cal S$ is a nonempty subset of $V_{\cal L}$ then we will denote by $\Gamma(\mathcal S)=(V_{\mathcal S},E_{\cal S})$ the subgraph of $\Gamma(\cal L)$ induced by $\cal S$. That is,
\begin{equation}
    \left\{
        \begin{aligned}
            & V_{\cal S}:= \mathcal{S}
            \\ 
            & E_{\cal S}:= \left\{\{x, y\} \in E_{\cal L} \, | \, x, y \in V_{\cal S}\right\}.
        \end{aligned}
    \right.
\end{equation}
Note that the subgraph of $\Gamma(\mathcal L)$ induced by any nonempty subset of $\{\lambda s \mid \lambda \in \mathbb F, s \in V_{\cal L}\}$ is connected.
\begin{remark}
    If the commuting graph of a noncommutative Lie algebra $\mathcal L$ is connected, then its diameter is not less than $2$, i.e., 
    \begin{equation}
        \diam(\Gamma(\mathcal L)) \geq 2.
    \end{equation}
   In the rest of the paper, all considering Lie algebras are always assumed to be noncommutative and therefore, the diameters of their commuting graphs, if connected, are not less than $2$.
\end{remark}

It is worth pointing out that the commuting graphs of two isomorphic Lie algebras are obviously isomorphic. However, the converse is not true in general. This is illustrated by the following example.

\begin{example} 
Let $\mathcal L_1$ and $\mathcal L_2$ be the two 3-dimensional Lie algebras defined over the real field $\mathbb R$ as follows:
     \begin{itemize}
        \item $\mathcal{L}_1 = \left< x_1, x_2, x_3\right>: [x_3,x_1] = x_1, \, [x_3,x_2]=x_2;$
        \item $\mathcal{L}_2 = \left< y_1, y_2, y_3\right>: [y_3,y_1] = y_1, \, [y_3,y_2] = y_1 + y_2.$
    \end{itemize}
According to the classification in \cite[Chapter 16]{Snob}, $\mathcal{L}_1$ is not isomorphic to $\mathcal{L}_2$. Now, it is easy to see that 
$
    Z(\mathcal L_1)=Z(\mathcal L_2)=\{0\}.
$
The vertex sets of $\Gamma(\mathcal L_1)$ and $\Gamma(\mathcal L_2)$ are therefore determined by:
\begin{align*}
    V_{\mathcal L_1}
    = \left\{ a_1x_1+a_2x_2+a_3x_3 \, | \, (a_1, a_2, a_3) \neq (0,0,0) \right\}, \\
    V_{\mathcal L_2}
    = \left\{ a_1y_1+a_2y_2+a_3y_3 \, | \, (a_1, a_2, a_3) \neq (0,0,0) \right\}.
\end{align*}
Let $f:V_{\mathcal L_1}
\rightarrow V_{\mathcal L_2}$
be a function defined by:
\begin{equation*}
    f(a_1x_1+a_2x_2+a_3x_3)=a_1y_1+a_2y_2+a_3y_3.
\end{equation*}
Then $f$ is a bijection.
Moreover, it is elementary to check that
\begin{equation*}
\begin{array}{ll}
    & [f(a_1x_1+a_2x_2+a_3x_3),f(b_1x_1+b_2x_2+b_3x_3)] = 0 \\
    \iff & 
    [a_1y_1+a_2y_2+a_3y_3,b_1y_1+b_2y_2+b_3y_3] = 0 \\
    \iff &
    a_1b_3-a_3b_1=a_2b_3-a_3b_2 = 0 \\
    \iff &
    [a_1x_1+a_2x_2+a_3x_3,b_1x_1+b_2x_2+b_3x_3] = 0.
\end{array}
\end{equation*}
Therefore, $f$ is an isomorphism of graphs $\Gamma(\mathcal L_1)$ and $\Gamma(\mathcal L_2)$.
\end{example}

Moreover, it is worth noting that the connectedness of the commuting graph of a Lie algebra substantially depends on the ground field. 
The example that follows demonstrates this.
\begin{example} 
    Let us consider a Lie algebra $\cal L$ defined over a field $\mathbb F$ by the following structure constants: \begin{equation}\label{Equationn63}
        \mathcal L = \left< e_1, e_2, \ldots,e_6\right> \colon  
            [e_1,e_3] = [e_2,e_4]=e_5, \, [e_1,e_4]=-[e_2,e_3]=-e_6.
    \end{equation}
    The center of $\cal $ is easily determined as follows $
        Z(\mathcal L) = \left< e_5, e_6 \right>.
    $
   We shall show that in the case of the real field $\mathbb F = \mathbb R$, $\Gamma(\mathcal L)$ is disconnected, whereas in another case, $\mathbb F = \mathbb F_2$, is connected.  Moreover, in the case where $\mathbb F = \mathbb F_2$, we will show that
    \begin{equation}
        \diam\left(\Gamma(\mathcal L)\right) = 3.
    \end{equation}
    Note that if the ground field is $\mathbb R$, the above Lie algebra $\mathcal L$ is isomorphic to $\mathfrak{n}_{6,3}$ classified in \cite[Chapter 19]{Snob}. From the structure constants of $\mathcal L$, we easily observe that
     \begin{align}
        [a_1e_1+a_2e_2, b_1e_1+b_2e_2+b_3e_3+b_4e_4] = 0 
        \iff  a_1b_3+a_2b_4=a_2b_3-a_1b_4=0 \label{Equationeq5}
    \end{align}
    We now consider the following two cases depending on the ground field as follows.
    \\[5pt]
    \noindent\textbf{Case 1.} $\mathbb F = \mathbb R$. Let $S= \{a_1e_1+a_2e_2+a_5e_5+a_6e_6\mid (a_1,a_2)\neq (0,0)\}$. Then $S$ is a subset of the vertex set of $\Gamma(\cal L)$. It is easy to see that for any $a_1, a_2 \in \mathbb R$ such that $a_1^2+a_2^2\neq 0$, we have
    \begin{equation}
        a_1b_3+a_2b_4=a_2b_3-a_1b_4=0 \iff b_3=b_4=0.
    \end{equation}
    Therefore, the subgraph of $\Gamma(\mathcal L)$ induced by $S$ is a connected component of $\Gamma(\mathcal L)$. It follows that $\Gamma(\mathcal L)$ is disconnected, as required.
    \\[5pt]
    \noindent \textbf{Case 2.} $\mathbb F = \mathbb F_2$. Then, by direct calculations, we may verify that the graph in \figurename~\ref{fig.1} is a subgraph of $\Gamma(\mathcal L)$.
    \begin{figure}
    \centering
        \begin{tikzpicture}[auto=center,every node/.style={,draw=black,fill=white}]
        \node (a1) at (-2,1){$e_1+e_2$};
        \node (a2) at (2,1){$e_3+e_4$};
        \node (a3) at (0,-1){$e_1+e_2+e_3+e_4$};
        \node (a11) at (-5,0){$e_2$};
        \node (a12) at (-5,2){$e_1$};
        \node (a13) at (-5,3){$e_1+e_3+e_4$};
        \node (a14) at (-5,-1){$e_2+e_3+e_4$};
        \node (a21) at (5,0){$e_4$};
        \node (a22) at (5,2){$e_3$};
        \node (a23) at (5,3){$e_1+e_2+e_3$};
        \node (a24) at (5,-1){$e_1+e_2+e_4$};
        \node (a31) at (-5,-3){$e_1+e_3$};
        \node (a32) at (-2,-3){$e_1+e_4$}; 
        \node (a33) at (2,-3){$e_2+e_3$};
        \node (a34) at (5,-3){$e_2+e_4$};
        \draw (a1) -- (a2) -- (a3) -- (a1);
        \draw (a1) -- (a11); \draw (a1) -- (a13);
        \draw (a1) -- (a12); \draw (a1) -- (a14);
        \draw (a2) -- (a21); \draw (a2) -- (a23);
        \draw (a2) -- (a22); \draw (a2) -- (a24);
        \draw (a3) -- (a31); \draw (a3) -- (a33);
        \draw (a3) -- (a32); \draw (a3) -- (a34);
        \draw (a13.south west) .. controls (-7,1) .. (a14.north west);
        \draw (a11.south west) .. controls (-6,1) .. (a12.north west);
        \draw (a23.south east) .. controls (7,1) .. (a24.north east);
        \draw (a21.south east) .. controls (6,1) .. (a22.north east);
        \draw (a31) .. controls (0,-5) .. (a34);
        \draw (a32) .. controls (0,-4) .. (a33);
    \end{tikzpicture}
    \caption{A subgraph of the commuting graph of the Lie algebra defined in Equation~\eqref{Equationn63} over $\mathbb F_2$.}\label{fig.1}
    \end{figure}
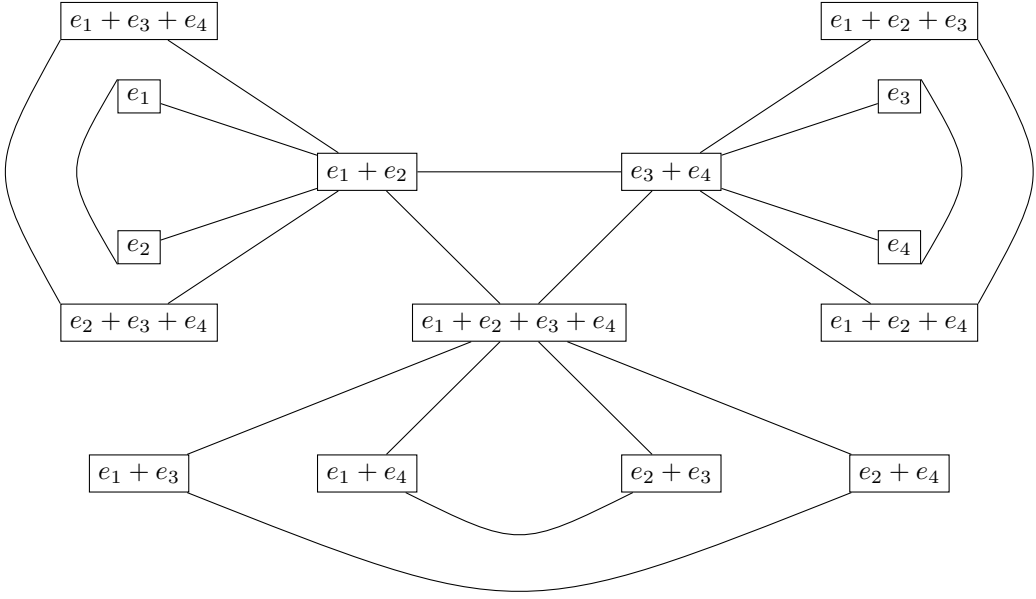
    Since the vertex set of $\Gamma(\mathcal L)$ is 
    \begin{equation}
        V_{\cal L} = \left\{ a_1e_1 + a_2e_2 + a_3e_3 + a_4e_4 + a_5e_5 + a_6e_6 \mid (a_1,a_2,a_3,a_4) \in \mathbb F_2^4 \setminus (0,0,0,0) \right\},
    \end{equation}
we can conclude from \figurename~\ref{fig.1} that $\Gamma(\mathcal L)$ is connected and 
\begin{equation}\label{Equation<=3}
    \diam(\Gamma(\mathcal L)) \leq 3.
\end{equation}
Furthermore, it can be verified that
\begin{equation}
    \left\{ 
    \begin{aligned}
        \relax [e_1+e_3,a_1e_1+a_2e_2+a_3e_3+a_4e_4] = 0  \\ 
        [e_3,a_1e_1+a_2e_2+a_3e_3+a_4e_4] = 0
    \end{aligned}
    \right. \iff a_1=a_2=a_3=a_4=0.
\end{equation}
Therefore, there is no vertex $v$ such that $\{e_1+e_3,v,e_3\}$ is a path. Equivalently, 
\begin{equation}
    \dist(e_1+e_3,e_3) \geq 3.
\end{equation}
Besides, $\{e_1+e_3,e_1+e_2+e_3+e_4,e_3+e_4,e_3\}$ is a path in $\Gamma(\mathcal L)$. It follows that  $\dist(e_1+e_3,e_3) = 3.$ In combining with Equation~\eqref{Equation<=3} we get $\diam(\Gamma(\mathcal L))=3$, as required.
    


\end{example}

\section{Connectedness and diameters of commuting graphs of general Lie algebras}\label{section:general}

In this section, we will study the diameters of commuting graphs of decomposable Lie algebras and give some equivalent conditions for a general Lie algebra to have a connected commuting graph.

\begin{proposition}\label{Prop:decomposable}
    Let $\cal L$ be a decomposable Lie algebra over a field $\mathbb{F}$.
    \begin{itemize}
        \item[(i)] \label{item(i)pro7}If $\mathcal{L}=\mathcal{L}_1\oplus \mathbb{F}^k$ for some noncommutative Lie subalgebra $\mathcal{L}_1 \subseteq \cal L$, then  $\Gamma(\cal L)$ is connected if and only if $\Gamma(\mathcal{L}_1)$ is connected. Moreover,    $\diam(\Gamma(\mathcal{L}))=\diam(\Gamma(\mathcal{L}_1)).$
        
        \item[(ii)] \label{item(ii)pro7} If $\mathcal{L}=\mathcal{L}_1 \oplus \mathcal{L}_2$ for some noncommutative Lie subalgebras $\mathcal{L}_1$ and $\mathcal{L}_2$ of $\cal L$, then $\Gamma(\cal L)$ is connected. Moreover, $\diam(\Gamma(\mathcal{L}))\leq 3.$ 
    \end{itemize}
\end{proposition}

\begin{proof}
    Obviously, $\mathcal L$ can be decomposed into either $\mathcal L_1 \oplus \mathbb F^k$ or $\mathcal L_1 \oplus \mathcal L_2$.
    \begin{itemize}
        \item[$(i)$] Assume that $\mathcal{L}=\mathcal{L}_1\oplus \mathbb{F}^k$ for some noncommutative subalgebra $\mathcal{L}_1 \subseteq \cal L$. It is then trivial to see that $Z(\mathcal{L})=Z(\mathcal{L}_1)\oplus \mathbb{F}^k$. Then the vertex set of the commuting graph of $\mathcal L$, $V_{\cal L}$, can be expressed via the vertex set of the commuting graph of $\mathcal L_1$, $V_{\mathcal L_1}$, as follows:
        \begin{equation}
            V_{\mathcal L} = 
            \{x + y\,
         | \, x \in V_{\mathcal{L}_1}, \, y \in \mathbb{F}^k\}.
        \end{equation}
        Moreover, for every $x_1, \ldots,x_k\in V_{\mathcal{L}_1}$ 
        and for every $y_1, \ldots, y_k\in\mathbb{F}^k$, we easily see that $\left\{ x_1+y_1, x_2+y_2, \ldots, x_k+y_k\right\}$ is a path of $\Gamma(\mathcal{L})$ if and only if $\{x_1, \ldots, x_k\}$ is a path of $\Gamma(\mathcal{L}_1)$. This proves the first part of the theorem.
        
        \item[$(ii)$] Assume that $\cal L$ is a direct sum of two noncommutative subalgebras $\mathcal{L}_1, \mathcal{L}_2$. 
        First, we show that
        \begin{equation}\label{eq19}
            Z(\mathcal{L})= Z(\mathcal{L}_1) \oplus Z(\mathcal{L}_2).
        \end{equation}
        Indeed, if $x \in Z(\cal L)$ then $x=x_1+x_2$ for some $x_1 \in \mathcal{L}_1, x_2 \in \mathcal{L}_2$ such that $[x_1+x_2,L]=0$. Since $[\mathcal{L}_1,\mathcal{L}_2]=0$, we have $[x_1,\mathcal{L}_1]=[x_2,\mathcal{L}_2]=0$. Therefore, $x \in Z(\mathcal{L}_1)+Z(\mathcal{L}_2)$. Besides, if $x \in Z(\mathcal{L}_1)+Z(\mathcal{L}_2)$ then there exist $x_1 \in Z(\mathcal{L}_1)$ and $x_2 \in Z(\mathcal{L}_2)$ such that $x=x_1+x_2$. Since $[\mathcal{L}_1,\mathcal{L}_2]=0$, we get $[x,\mathcal{L}_1]=[x,\mathcal{L}_2]=0$ which turns out that $[x,\mathcal{L}]=0$ or $x \in Z(\mathcal{L})$. This proves Equation~\eqref{eq19}.
        
        

        
        
        Now, let $x_1, x_2$ be two arbitrary vertices of $\Gamma(\mathcal{L})$. Then there exist $x_{11}, x_{21}\in \mathcal{L}_1$, and $x_{12},x_{22} \in \mathcal{L}_2$ such that $x_1=x_{11}+x_{12}$ and $x_2=x_{21}+x_{22}$. Note that the vertex sets of both the commuting graphs of $\mathcal L_1$ and $\mathcal L_2$ are non-empty due to the fact that $\mathcal L_1$ and $\mathcal L_2$ are noncommutative.
        \begin{itemize}
                      
            \item If $x_{11} \in Z(\mathcal{L}_1)$ and $x_{21} \in Z(\mathcal{L}_1)$ then for every $t \in V_{\mathcal{L}_1}$, we have the path $\{x_1,t,x_2\}$. 

            \item If $x_{11} \in Z(\mathcal{L}_1)$ and $x_{21} \in V_{\mathcal{L}_1}$
            then $\{x_1,x_{21},x_2\}$ is a path.

            \item Similarly, if $x_{22} \in Z(\mathcal{L}_2)$, then $x_1$ and $x_2$ are joined by a path of length at most 2.  

            \item If $x_{11} \in V_{\mathcal{L}_1}$ 
            and $x_{22} \in V_{\mathcal{L}_2}$  
            then we can derive from Equation~\eqref{eq19} that both $x_{11}$ and $x_{22}$ are the vertices of $\Gamma(\mathcal{L})$. If so, it is elementary to see that $\{x_1,x_{11},x_{22},x_2\}$ is a path in $\Gamma(\mathcal L)$.
            
        \end{itemize}
        In summary, we have proven that if $\mathcal{L}$ is a direct sum of two noncommutative subalgebras then every pair of distinct vertices $x_1,x_2$ are joined by a path of length at most $3$. This proves the second part of the theorem.
    \end{itemize}
\end{proof}

We shall give here an example to demonstrate Proposition~\ref{Prop:decomposable}.
\begin{example} \label{example:n21}
    Let $\mathcal{L} = \left< e_1, e_2, e_3,e_4\right>$ be the direct sum of two copies of the Lie algebra of dimension 2 over the binary field $\mathbb F_2$. That is, $\mathcal L$ is the $4$-dimensional Lie algebra defined over $\F_2$  as follows:
    \begin{equation}\label{eq:example}
            \mathcal{L} = \left< e_1, e_2, e_3,e_4\right> \colon [e_1,e_2] = e_2, \, [e_3,e_4]=e_4.
    \end{equation}
    Then we have $Z(\mathcal{L})=\{0\}$. Therefore, $\Gamma(\mathcal L)$ has 15 vertices: 
    \begin{equation}
        V_{\mathcal L} =\{e_1, \ldots, e_4, e_1+e_2, \ldots, e_3+e_4, e_1+e_2+e_3, \ldots, e_2+e_3+e_4, e_1+e_2+e_3+e_4\}.
    \end{equation}
    According to the structure constants of $\mathcal L$, we easily observe that 
    \begin{equation}
    [a_1e_1+a_2e_2+a_3e_3+a_4e_4,b_1e_1+b_2e_2+b_3e_3+b_4e_4] = 0 \iff a_1b_2-a_2b_1=a_3b_4-a_4b_3=0.
    \end{equation}
    Thus, we obtain the commuting graph of $\mathcal L$ as in \figurename~\ref{fig.2}.
    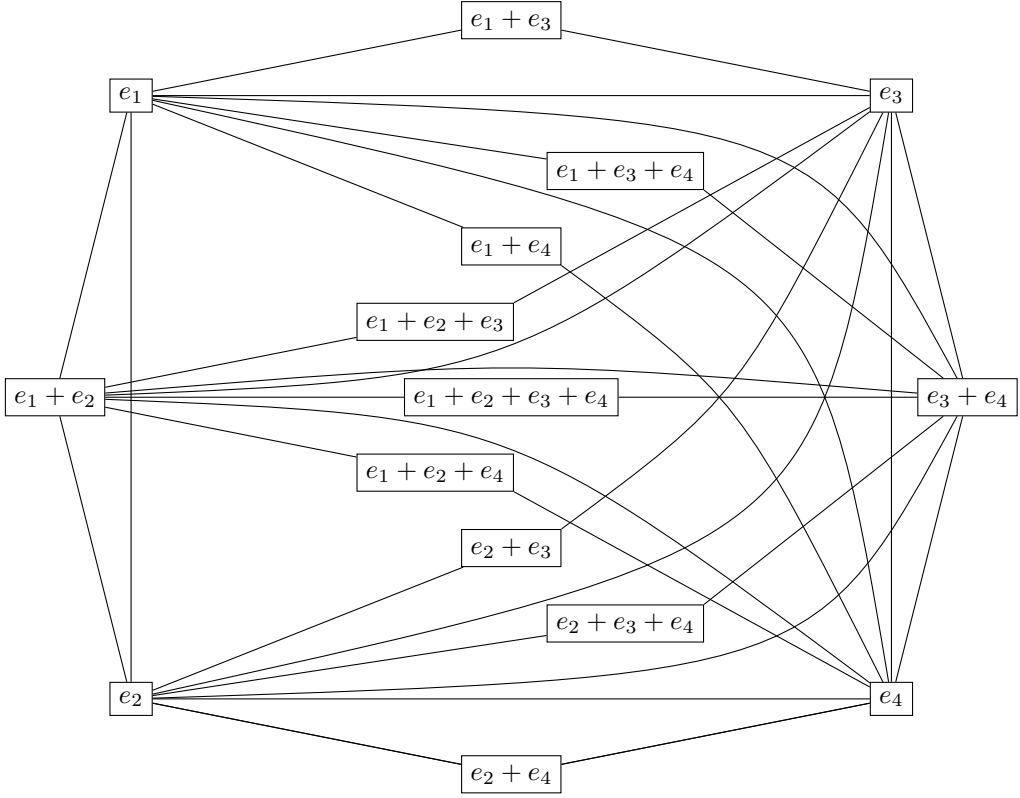
\begin{figure}[!ht]
    \centering
        \begin{tikzpicture}[auto=center,every node/.style={,draw=black,fill=white}]
        \node (e1) at (-5,4){$e_1$};
        \node (e2) at (-5,-4){$e_2$};
        \node (e12) at (-6,0){$e_1+e_2$};
        \node (e3) at (5,4){$e_3$};
        \node (e4) at (5,-4){$e_4$};
        \node (e34) at (6,0){$e_3+e_4$};
        \node (e13) at (0,5){$e_1+e_3$};
        \node (e134) at (1.5,3){$e_1+e_3+e_4$};
        \node (e14) at (0,2){$e_1+e_4$};
        \node (e123) at (-1,1){$e_1+e_2+e_3$};
        \node (e1234) at (0,0){$e_1+e_2+e_3+e_4$};
        \node (e124) at (-1,-1){$e_1+e_2+e_4$};
        \node (e23) at (0,-2){$e_2+e_3$};
        \node (e234) at (1.5,-3){$e_2+e_3+e_4$};
        \node (e24) at (0,-5){$e_2+e_4$};
        \draw (e1) -- (e12) -- (e2) -- (e1);
        \draw (e3) -- (e34) -- (e4) -- (e3);
        \draw (e1) -- (e13) -- (e3) -- (e1);
        \draw (e2) -- (e24) -- (e4) -- (e2);
        \draw (e1) -- (e134); \draw (e134.south east) -- (e34); \draw (e1) .. controls (4,3.7) .. (e34);
        \draw (e1) -- (e14); \draw (e14.south east) .. controls (3,0) .. (e4); \draw (e1) .. controls (4,2) .. (e4); 
        \draw (e12) -- (e123); \draw (e123.north east) -- (e3); \draw (e12) .. controls (0,0.25) .. (e3);
        \draw (e12) -- (e1234) -- (e34); 
        \draw (e12) -- (e124); \draw (e124.south east) -- (e4);
        \draw (e2) -- (e23); \draw (e23.north east)  .. controls (3,0) .. (e3); \draw (e2) .. controls (4,-2) .. (e3); 
        \draw (e2) -- (e234); \draw (e234.north east) -- (e34); \draw (e2) .. controls (4,-3.7) .. (e34);
        \draw (e2) -- (e24) -- (e4); \draw (e12) .. controls (0,-0.25) .. (e4);
        \draw (e12) .. controls (0,0.5) .. (e34);
        %
    \end{tikzpicture}
    \caption{The commuting graph of the Lie algebra $\mathcal L$ defined in Equation \eqref{eq:example} with the ground field $\mathbb F_2$.}\label{fig.2}.
    \end{figure}
    It is clear that the diameter of the graph in \figurename~\ref{fig.2} is 3. Therefore,
   \begin{equation}\label{example.equality}
       \diam\left(\Gamma(\mathcal{L})\right)=3.
   \end{equation}
\end{example}

\begin{remark}\label{remark}
    The equality~\eqref{example.equality} still holds for any field $\mathbb F$.  Indeed, 
    \begin{equation}
        \left\{ 
            \begin{aligned}
                & [e_1+e_3, a_1e_1+a_2e_2+a_3e_3+a_4e_4]=0,\\
                & [e_2+e_4, a_1e_1+a_2e_2+a_3e_3+a_4e_4]=0
            \end{aligned}
        \right. \iff \left\{ 
            \begin{aligned}
                & a_2=a_4=0,\\
                & a_1=a_3=0.
            \end{aligned}
        \right.
    \end{equation}
    Therefore, there is no vertex in $\Gamma(\mathcal L)$ which is adjacent to both $e_1+e_3$ and $e_2+e_4$. It follows that $\diam(\Gamma(\mathcal L)) \geq 3.$ Applying Proposition~\ref{Prop:decomposable}, we observe that $\diam(\Gamma(\mathcal L)) = 3$, as required.
\end{remark}
\begin{proposition}\label{Prop.disconnected-setS}
    Let $\cal L$ be an indecomposable Lie algebra. Then $\Gamma(\cal L)$ is disconnected if and only if there is a nonempty set $\mathcal{S} \subsetneq V_{\cal L}$ 
    such that $\ker (\ad_x) \subseteq \mathcal{S} \cup Z(\cal L)$ for every $x \in \cal S$.
\end{proposition}

\begin{proof} 
    Suppose that there is a nonempty set $\mathcal{S} \subsetneq V_{\cal L}$ 
    such that $\ker (\ad_x) \subseteq \mathcal{S} \cup Z(\cal L)$ for every $x\in \mathcal{S}$. Let $y$ be any vertex in $\Gamma(\cal L)$ which does not belong to $\mathcal{S}$. Such $y$ exists because $\mathcal{S} \subsetneq V_{\cal L}$. Since $\ker (\ad_x) \subseteq \mathcal{S} \cup Z(\cal L)$ for every $x\in \mathcal{S}$, we have $y \notin \ker (\ad_x)$ for every $x\in \mathcal{S}$. Therefore, $[y,x]\neq 0$ for every $x\in \mathcal{S}$. It follows that $y$ is not adjacent to any vertex in $\mathcal{S}$. Thus, there is no path between any pair of vertices $x\in \mathcal{S}$ and $y\in V_{\mathcal{L}} \setminus \mathcal{S}$. Therefore, $\Gamma(\cal L)$ is disconnected.

    Conversely, suppose that $\Gamma(\cal L)$ is disconnected. Let $G$ be a connected component of $\Gamma(\cal L)$. Let $\mathcal{S}=V_{G}$ be the vertex set of $G$. Then $\mathcal{S} \subsetneq V_{\cal L}$. 
    Let $x$ be any element in $\cal S$ and let $y$ be any element in $\ker (\ad_x)$. Then $[x,y]=0$. Thus, $x$ is adjacent to any element in $\ker (\ad_x) \setminus Z(\cal L)$. Since $G$ is a connected component of $\Gamma(\cal L)$, we should have $\ker (\ad_x) \setminus Z(\mathcal{L}) \subseteq \cal S$. In other words, $\ker (\ad_x) \subseteq \mathcal S \cup Z(\cal L)$ for every $x \in \cal S$. This completes the proof. 
\end{proof}

In particular, we may obtain the connectedness of the commuting graph of a Lie algebra $\mathcal L$ by examining a basis of $\mathcal L$ as in Corollary~\ref{corollary} and the dimension of $\ker(\ad_{x})$ for some $x \in \mathcal L$ as in Corollary~\ref{coro1}.
\begin{corollary}\label{corollary}
    Let $\cal L$ be an indecomposable, finite-dimensional Lie algebra. Suppose that $Z^c(\cal L)$ is a complementary subspace of $Z(\cal L)$ in $\cal L$ and 
    $\boldsymbol{b}$ is a basis of $Z^c(\cal L)$. Moreover, assume that there is a nonempty proper subset $\boldsymbol{b}_1$ of $\boldsymbol{b}$ such that $[x_1,x_2]\neq 0$ for every $x_1\in \left<\boldsymbol{b}_1\right> \setminus \{0\}$ and $x_2 \in \left<\boldsymbol{b}\setminus\boldsymbol{b}_1\right> \setminus \{ 0 \}$. Then $\Gamma(\cal L)$ is disconnected.
\end{corollary}

\begin{proof}
    Let $\mathcal{S}=\left\{\sum_{x\in \boldsymbol{b}_1} \lambda_x x+y \, | \,\lambda_x \in \mathbb F, y\in Z(\mathcal{L})\right\} \setminus Z(\cal L)$. Then $\emptyset \subsetneq \mathcal{S} \subsetneq V_{\mathcal{L}}$ and $\ker (\ad_x) \subseteq \mathcal{S} \cup Z(\cal L)$ for every $x\in \cal S$. According to Proposition~\ref{Prop.disconnected-setS}, we conclude that $\Gamma(\cal L)$ is disconnected, as required.
\end{proof}

\begin{corollary}\label{coro1}
    Let $\cal L$ be an indecomposable, finite dimensional Lie algebra with center $Z(\cal L)$. if there exists $x$ such that $\dim\left(\ker(\ad_x ) \right) = \dim Z(\mathcal{L})+1$ then $\Gamma(\mathcal L)$ is disconnected.
\end{corollary}

\begin{proof}
    Since $\dim \left(\ker (\ad_x)\right) = \dim Z(\mathcal L)+1$, we have $x\notin Z(\mathcal L)$. Indeed, if $x\in Z(\mathcal L)$ then $\dim Z(\mathcal L)=\dim (\ker(\ad_x)) - 1 = \dim \mathcal L-1$ which turns out that $\mathcal L$ is commutative, a contradiction. Therefore, the dimension of the direct sum of two vector spaces $Z(\mathcal L)$ and $\left< x\right>$ can be determined by the following formula:
    \begin{equation}
        \dim\left(Z(\mathcal L) + \left< x \right>\right) = \dim Z(\mathcal L)+1.
    \end{equation}
    Furthermore, it is clear that $Z(\mathcal L) + \left< x \right>$ is contained in the vector space $\ker (\ad_x)$. We then conclude from the equality $\dim\left(\ker(\ad_x)\right)  = \dim Z(\mathcal{L})+1$ that 
    \begin{equation}\label{eq:cor11}
        \ker (\ad_x) = Z(\mathcal L) + \left< x \right>.
    \end{equation}
    Besides, it is easy to see that $\ker (\ad_{ax+y})= \ker (\ad_x)$ for every $0 \neq a\in \mathbb F$ and $y \in Z(\mathcal L)$. Therefore, by letting $\mathcal S=\left\{ax+y \, | \, 0\neq a\in \mathbb F, y \in Z(\mathcal L) \right\}$, we observe from Equality~\eqref{eq:cor11} that 
    $
        \ker (\ad_s) = Z(\mathcal L) \cup \mathcal S$ for every  $s \in \mathcal S$.
    Applying Proposition~\ref{Prop.disconnected-setS}, we get the proof.
\end{proof}

\begin{example} 
    Let $\mathfrak{n}_{6,3}$ be the real Lie algebra defined in \cite[Chapter 19]{Snob} as follows:
    \begin{equation}
        \mathfrak{n}_{6,3} = \left< e_1, e_2, \ldots,e_6\right> \colon 
            [e_1,e_3] = [e_2,e_4]=e_5, \, [e_1,e_4]=-[e_2,e_3]=-e_6.
    \end{equation}
    Since $Z(\mathfrak{n}_{6,3}) = \left< e_5, e_6 \right>$, a complementary of $Z(\mathfrak{n}_{6,3})$ in the vector space $\mathfrak n_{6,3}$ can be chosen to be $Z^c(\mathfrak{n}_{6,3}) = \left< e_1, e_2, e_3, e_4 \right>$.
     Let $\boldsymbol{b}_1=\{e_1, e_2\}$. Then 
    \begin{equation}
        [\lambda_1e_1+\lambda_2 e_2,\mu_1 e_3+\mu_2 e_4]=0 \iff \lambda_1\mu_1+\lambda_2\mu_2=\lambda_1\mu_2-\lambda_2\mu_1=0 \iff \left[
    \begin{aligned}
        \lambda_1=\lambda_2=0, \\
        \mu_1=\mu_2=0.
    \end{aligned}
    \right.
    \end{equation}
    Therefore, by applying Corollary~\ref{corollary}, $\Gamma(\mathfrak n_{6,3})$ is disconnected. 
\end{example}

By definition, the vertex set of the commuting graph of a Lie algebra $\cal L$ is $V_{\cal L} = \mathcal L \setminus Z(\mathcal{L})$. Therefore, when the center of $\mathcal L$ is ``big enough", the commuting graph $\Gamma(\mathcal L)$ should be disconnected. This remarkable point is given in the following proposition. 
\begin{proposition}\label{Prop.disconnected}
    Let $\cal L$ be an indecomposable, finite-dimensional Lie algebra over a field $\mathbb F$ such that $\dim \mathcal L - \dim Z(\mathcal{L}) \leq 3$. Then $\Gamma(\cal L)$ is disconnected.
\end{proposition}

\begin{proof}
    Since $\mathcal L$ is noncommutative, we must have $\dim \mathcal L - \dim Z(\mathcal L) \geq 2$. Therefore, $\dim \mathcal L - \dim Z (\mathcal L) \in \{2, \, 3\}$. 
\\[5pt]
    \noindent
    \textbf{Case 1.} $\dim \mathcal L - \dim Z(\mathcal L)=2$. If so, there are $e_1, e_2 \in \mathcal L$ such that the vector space $\mathcal L$ is decomposed into the following direct sum: $\mathcal L= Z(\mathcal L) + \left< e_1, e_2\right>$. It follows that the vertex set of $\Gamma(\mathcal L)$ is 
    \begin{equation}
        V_{\cal L}=\left\{ a_1e_1+a_2e_2+x \, | \, (0,0) \neq (a_1, a_2) \in \mathbb F^2, \, x \in Z(\mathcal L) \right\}.
    \end{equation}
    Since $\mathcal L$ is noncommutative, we have $[e_1, e_2] \neq 0$. Therefore, for every $x,y \in Z(\mathcal L)$ and for every $a_1, a_2, b_1, b_2 \in \mathbb F$, we have
    \begin{equation}
        [a_1e_1+a_2e_2+x,b_1e_1+b_2e_2+y]=0 \iff a_1b_2-a_2b_1=0.
    \end{equation}
    Thus, $\left\{a_1e_1+x \, | \, 0\neq a_1 \in \mathbb F, \, x\in Z(\mathcal L) \right\}$ and $\left\{a_2e_2+y \, | \, 0 \neq a_2 \in \mathbb F, \, y \in Z(\mathcal L) \right\}$ are two different connected components of $\Gamma(\mathcal L)$. It follows that $\Gamma(\mathcal L)$ is disconnected, as required.
\\[5pt]
\noindent
    \textbf{Case 2.} $\dim \mathcal L - \dim Z(\mathcal L)=3$. Without loss of generality, we assume that the vector space $\mathcal L$ has a decomposition $\mathcal L = Z(\mathcal L) + \left< e_1, e_2, e_3 \right>$ for some independent vectors $\{e_1, e_2, e_3\}$ in $\mathcal L$. If so, the derived algebra of $\mathcal L$, denoted by $\mathcal L^1$, is determined by:
    \begin{equation}
        \mathcal L^1 = \left< [e_1, e_2], [e_1,e_3], [e_2,e_3] \right>.
    \end{equation}
    First, we prove that $\dim \mathcal L^1 \geq 2$. By contradiction, suppose that $\dim \mathcal L^1=1$. Then we may assume (after a re-indexing if necessary) that $\mathcal{L}^1=\left< [e_1,e_2] \right>$. Thus, there are $a$ and $b$ in $\mathbb F$ such that $[e_1,e_3]=a[e_1,e_2]$ and $[e_2,e_3]=b[e_1,e_2]$. It follows that 
    \begin{equation}
        [e_1,e_3-ae_2+be_1]=[e_2,e_3-ae_2+be_1]=[e_3,e_3-ae_2+be_1]=0.
    \end{equation}
    This turns out that $e_3-ae_2+be_1$ belongs to $Z(\mathcal L)$, a contradiction.
    Therefore, $\dim \mathcal L^1 \geq 2$. Without loss of generality, we may assume that $[e_1,e_2]$ and $[e_1,e_3]$ are linearly independent. Thus, for every $x\in Z(\mathcal L)$, we have
    \begin{equation}
        [e_1, a_1e_1+a_2e_2+a_3e_3+x] = 0 \iff a_2=a_3=0.
    \end{equation}
    In other words, $\ker (\ad_{e_1}) = Z(\mathcal L) + \left< e_1\right>$. According to Corollary~\ref{coro1}, we observe that $\Gamma(\mathcal L)$ is disconnected, as required.
\end{proof}
The following corollary follows immediately from Propositions~\ref{Prop:decomposable} and ~\ref{Prop.disconnected}.
\begin{corollary}\label{cor.dim<=3}
    Let $\mathcal L$ be a Lie algebra of dimension not exceeding $3$. Then the commuting graph of $\mathcal L$ is disconnected.
\end{corollary}
\section{A Process to determine the Connectedness of the Commuting Graph of a Noncommutative Lie Algebra} \label{sec:algorithm}




In this section, we will give a process to determine the connectedness, as well as an upper bound for the diameter of the commuting graph of any Lie algebra $\cal L$. 
Let ${S}$ be a subset of the vertex set $V_{\cal L}$ of $\Gamma(\mathcal L)$, the {\it closed neighborhood} of ${S}$, denoted by $N({S})$, is defined as the union of ${S}$ with the set of all vertices of $\Gamma(\mathcal L)$ adjacent to a vertex in ${S}$. That is,
\begin{equation*}
    N({S}):= 
    \left\{x \in \mathcal L \setminus Z(\mathcal L) \, | \, [x,s]=0 \text{ for some } s \in {S}\right\}.
\end{equation*}


\begin{definition}
    Let $\cal L$ be a noncommutative Lie algebra and let ${S}$ be a non-empty subset of the vertex set $V_{\cal L}$. For each natural number $k\geq 1$, we define by induction, the $k$-th closed neighborhood of $S$ as: $N^{k}(S) := N\left(N^{k-1}(S)\right)$ where $N^0(S):=S$. The series of neighborhoods of ${S}$ is then defined as follows:
    \begin{equation}
        N^0({S}) \subseteq N^1({S}) \subseteq N^2({S}) \subseteq N^3({S}) \subseteq N^4({S}) \subseteq \cdots
    \end{equation}
    If there is a natural number $k$ such that
    \begin{equation}
        N^{k-1}(S) \subsetneq N^{k}(S) = N^{k+1}(S),
    \end{equation}
    we say that the series of neighborhoods of $S$ terminates at $k$.
\end{definition}
For convenience, in the case where $S$ contains only one vertex, say ${S}:= \{s\}$, the $k$-th neighborhood of $S$ will be also denoted as $N^k(s)$. And the series of neighborhoods of $\{s\}$ is defined as a consequence. 



\begin{theorem} \label{theorem:diam<d}
    Let $\cal L$ be a Lie algebra and let ${S}$ be a non-empty subset of the vertex set $V_{\cal L}$ such that the subgraph of $\Gamma(\mathcal L)$ induced by $S$, denoted by $\Gamma(S)$, is connected. 
    Suppose that the series of closed neighborhoods of ${S}$ terminates at 
    a positive integer $d$. Then $\Gamma(\mathcal L)$ is connected if and only if $N^d({S})=V_{\cal L}$. Moreover, if $N^d({S})=V_{\cal L}$ then $\diam(\Gamma(\mathcal L)) \leq 2d + \diam(\Gamma(S))$.
\end{theorem}

\begin{proof} 
\begin{itemize}
    \item[$(\Rightarrow)$] Assume that $N^d({S})=V_{\cal L}$.  Let $u, v$ be two distinct vertices in $V_{\cal L}$. We need to prove that there is path joining $u$ and $v$ having distance at most $2d+\diam(\Gamma(S))$. Indeed, let $k$ and $l$ be the smallest indices such that $u \in N^k(S)$ and $v\in N^l(S)$, respectively. Obviously, $k$ and $l$ are both upper bounded by $d$. Because $k$ is the smallest index such that $u\in N^k(S)$, we have $u \in N^{k}(S) \setminus N^{k-1}(\mathcal{S})$. If so, there is a vertex $u_1 \in N^{k-1}(S)$ such that $u$ and $u_1$ are adjacent. Moreover, we observe that $u_1$ must be an element in $N^{k-1}(S) \setminus N^{k-2}(S)$ (otherwise, $u \in N^{k-1}(S)$, a contradiction to the definition of $k$). Therefore, $\Gamma(\mathcal L)$ contains a path $\{u, \, u_1, \, \ldots, \, u_k\}$ where $u_k \in S$. 

    Similarly, $\Gamma(\mathcal L)$ contains a path $\{v, v_1, \ldots, v_l\}$ with $v_l\in S$. Since $\Gamma(S)$ is connected, there is a path joined $u_k$ and $v_l$. Therefore, $u$ and $v$ are joined by a path of length bounded above by $k+l+\diam(\Gamma(S)) \leq 2d+\diam(\Gamma(S))$, as required.

 
    \item[$(\Leftarrow)$] Assume that $N^d({S})\subsetneq V_{\cal L}$. We need to prove that $\Gamma(\mathcal L)$ is disconnected. By contradiction, suppose that $\Gamma(\mathcal L)$ is connected. Then, for any $u \in N^d({S})$ and $v \in V_{\cal L} \setminus N^d({S})$, there is a path $P: = \{u, u_1,\ldots, u_k, v\}$ in $\Gamma(\mathcal L)$. 
    It follows that  $v \in N^{d+k+1}(S)$. Since the series of closed neighborhoods of ${S}$ terminates at $d$, we obtain $v\in N^d(S)$, a contradiction. 
    So, $\Gamma(\mathcal L)$ is disconnected, as required.
\end{itemize} 
\end{proof}

\begin{corollary}\label{cor.<=2d}
    Let $\cal L$ be a Lie algebra and let $s$ be a vertex of its commuting graph. 
    Suppose that the series of closed neighborhoods of $s$ terminates at 
    a positive integer $d$. Then $\Gamma(\mathcal L)$ is connected if and only if $N^d({s})=V_{\cal L}$. Moreover, if $N^d({s})=V_{\cal L}$ then $\diam(\Gamma(\mathcal L)) \leq 2d$.
\end{corollary}

\begin{example}
    Let $\mathcal L=\mathrm{M}_{2}(\mathbb F)$ be the usual Lie algebra of all $2\times 2$ matrices over the field $\mathbb F$. Then 
    \begin{equation}
        Z(\mathcal L)=\{\lambda I:\lambda \in \mathbb F\}.
    \end{equation}
    Moreover, by letting $A=\begin{bmatrix} 1 & 0 \\ 0 & 0 \end{bmatrix}$, we can check easily that 
    \begin{equation}
        \left[ \begin{bmatrix}
            a_{11} & a_{12} \\ a_{21} & a_{22}
        \end{bmatrix},  A \right] = 0 \iff 
        \begin{bmatrix}
            a_{11} & a_{12} \\ a_{21} & a_{22}
        \end{bmatrix} A 
        - A \begin{bmatrix}
            a_{11} & a_{12} \\
            a_{21} & a_{22}
            \end{bmatrix} = 0
            \iff 
            a_{12}=a_{21}=0.
    \end{equation}
    Therefore,
    \begin{align}
        N^1(A)  = \left\{\begin{bmatrix}
            a_{11} & 0 \\ 0 & a_{22}
        \end{bmatrix} \, | \, a_{11}, \, a_{22} \in \R; \, a_{11} \neq a_{22}\right\}.
    \end{align}
    Furthermore, it is elementary to check that
    \begin{equation}
        \begin{bmatrix}
            b_{11} & b_{12} \\ b_{21} & b_{22}
        \end{bmatrix} 
        \begin{bmatrix}
            a_{11} & 0 \\ 0 & a_{22}
        \end{bmatrix}
        - \begin{bmatrix}
            a_{11} & 0 \\ 0 & a_{22}
        \end{bmatrix}
        \begin{bmatrix}
            b_{11} & b_{12} \\
            b_{21} & b_{22}
        \end{bmatrix} = 0
        \iff 
        \left\{
        \begin{array}{ll}
            (a_{11}-a_{22})b_{12} =0,  \\
            (a_{11}-a_{22})b_{21} =0.
        \end{array}
        \right.
    \end{equation}
    It thus follows that
    \begin{equation}
        N^2(A)=N^1(A).
    \end{equation}
    In the other words, the series of neighborhoods of $A$ terminates at $k=1$. Moreover, $N^1(A) \subsetneq V_{\cal L}$. By Theorem~\ref{theorem:diam<d}, we conclude that the commuting graph of $\mathrm{M}_2(\mathbb F)$ is disconnected.

    Remark this result was proven in 2004 by Akbari \cite{Akbari2004}. Akbari also showed that the commuting graph of $\mathrm{M}_2(\mathbb F)$ is disconnected for every field $\mathbb F$ and each connected component of $\Gamma(\mathrm{M}_2(\mathbb F))$ is a complete graph. Moreover, in the case whether $\mathbb F$ is a finite field, the number of connected component if $|\mathbb F|^2+|\mathbb F|+1$, and its connected component is of the size $|\mathbb F|^2-|\mathbb F|$.
\end{example}  

Based on Corollary~\ref{cor.<=2d}, we propose a process to determine the connectedness of $\Gamma(\mathcal L)$ and an upper bound for the diameter of its commuting graph.

\begin{algorithm}[H] 
	\begin{algorithmic}[1]
\Require{The structure constants of the Lie brackets}
 \Ensure{The connectedness and an upper bound for the diameter of the commuting graph} 
 
\textbf{Step 1:} Set ${S}:=\{\text{ }\}$ and $d:=0$. Choose any vertex in the commuting graph, that is, an element $x \in \mathcal L \setminus Z(\mathcal L)$ and append this element to  $S$.

\textbf{Step 2:} Use the structure constants to determine the closed neighborhood of ${S}$:
\begin{equation}
    N({S})=\{y \in \mathcal L \setminus Z(\mathcal L): [y,s]=0; \, \text{ for some } s \in {S}\}.
\end{equation}

\textbf{Step 3:} 
\begin{itemize}
    \item If $N({S})={S} \subsetneq \mathcal L \setminus Z(\mathcal L)$ then return ``disconnected".

    \item  If $N({S})=\mathcal L \setminus Z(\mathcal L)$ then return ``connected" and upper bound $2(d+1)$.

    \item  If ${S} \subsetneq N({S}) \subsetneq \mathcal L \setminus Z(\mathcal L)$ then append $N({S})$ to ${S}$, set $d:=d+1$, and repeat from \textbf{Step 2}.
\end{itemize}


\end{algorithmic}
	\caption{{\bf ConnectednessDiameter} (Determining the connectedness and an upper bound for the diameter of the commuting graph of a Lie algebra $\mathcal L$)}
	\label{algorithm}  
\end{algorithm}

\begin{remark}
    The process may not stop if we are not sure that the series of neighborhoods of ${\cal S}$ terminates. However, if $\cal L$ is finite-dimensional  over a finite field $\mathbb F$, then the commuting graph is finite, and therefore Process \ref{algorithm} stops at a finite iteration.
\end{remark}

\begin{example}\label{examplen411}
    Let $\mathfrak s_{4,11}$ be a Lie algebra defined over a field $\mathbb F$ as follows:
    \begin{equation}
        \mathfrak s_{4,11} = \left< e_1, e_2, e_3, e_4 \right> \colon [e_2,e_3]=e_1, \, [e_4,e_1] = e_1, \, [e_4,e_2]=e_2.
    \end{equation}
    It is easy to see that the the center of $\mathfrak s_{4,11}$ is just the trivial vector space. Therefore, each non-zero vector of 
    $\mathfrak s_{4,11}$ is a vertex of the corresponding commuting graph. Choose the vertex $e_4$ as the starting point to apply Process \ref{algorithm}. 
    Since
    \begin{equation}
        [a_1e_1+a_2e_2+a_3e_3+a_4e_4,e_4]= 0 \iff a_1=a_2=0,
    \end{equation}
    we have
    \begin{equation}
        N^1(e_4)=\left\{ a_3e_3+a_4e_4 \mid (a_3,a_4) \neq (0,0)\right\}.
    \end{equation}
    Besides, it is elementary to check that
    \begin{equation}\label{eq.ex19.41}
        [x_1e_1+x_2e_2+x_3e_3+x_4e_4,a_3e_3+a_4e_4] = 0 \iff -x_1a_4+x_2a_3=x_2a_4=0.
    \end{equation}
    Remark that
    \begin{equation}
        -x_1a_4+x_2a_3=x_2a_4=0 \iff \left[
        \begin{array}{ll}
            (x_1=0,x_2=0), & \text{if }  a_4 \neq 0, \\
            (x_1=x_1,x_2=0), & \text{if } a_3 \neq 0 = a_4, \\
            (x_1=x_1,x_2=x_2), & \text{if } a_3 = 0 = a_4. \\
        \end{array}
        \right.
    \end{equation}
    Therefore, we observe that
    \begin{equation}\label{eq.N2(e4)}
        N^2(e_4)=\left\{ a_1e_1+a_3e_3+a_4e_4 \mid (a_1, a_3, a_4) \neq (0,0,0)\right\}
    \end{equation}
    Similarly, we also have
    \begin{equation}
        [x_1e_1+x_2e_2+x_3e_3+x_4e_4,a_1e_1+a_3e_3+a_4e_4] = 0 \iff -x_1a_4+x_2a_3+x_4a_1=x_2a_4=0.
    \end{equation}
    And for any $(a_1,a_3,a_4) \neq (0,0,0)$, we may see that
    \begin{equation}\label{eq.ex19.45}
        \left\{
        \begin{array}{r}
            -x_1a_4+x_2a_3+x_4a_1=0 \\
            x_2a_4=0 
        \end{array} \right. \iff \left[
        \begin{array}{ll}
            (x_1=\dfrac{x_4a_1}{a_4},x_2=0,x_4=x_4), & \text{if }  a_4 \neq 0,  \\
            (x_1=x_1,x_2=x_2,x_4=\dfrac{-x_2a_3}{a_1}), & \text{if }  a_1 \neq 0 = a_4, \\
            (x_1=x_1,x_2=\dfrac{-x_4a_1}{a_3},x_4=x_4), & \text{if }  a_3 \neq 0= a_4, \\
        \end{array}
        \right.
    \end{equation}
    Therefore, $x_1e_1+x_2e_2+x_3e_3+x_4e_4$ is adjacent to $e_1-\dfrac{x_4}{x_2}e_3$ for any $(x_1,x_2,x_3,x_4)$ with $x_2 \neq 0$. Combining with Equation~\eqref{eq.N2(e4)}, we conclude that 
    \begin{equation}
        N^3(e_4)=\Gamma\left( \mathfrak s_{4,11} \right).
    \end{equation}
    According to Theorem~\ref{theorem:diam<d}, we observe that $\Gamma(\mathfrak s_{4,11})$ is connected and 
    \begin{equation}\label{eq.<=6}
        \diam(\Gamma(\mathfrak s_{4,11})) \leq 6.
    \end{equation}
    We now prove that the diameter of $\Gamma(\mathfrak s_{4,11})$ is equal to $3$. Indeed, it is easy to verify that
    \begin{equation}
        \left\{
            \begin{array}{ll}
                [x_1e_1+x_2e_2+x_3e_3+x_4e_4,e_2] = 0  
                \\
                {}
                [x_1e_1+x_2e_2+x_3e_3+x_4e_4,e_4] = 0
            \end{array}
        \right. \iff 
        \left\{
            \begin{array}{ll}
                x_3e_1+x_4e_2=0  \\
                 x_1e_1+x_2e_2= 0
            \end{array}
        \right. \iff 
        \left\{
            \begin{array}{ll}
                x_3=x_4=0  \\
                 x_1=x_2= 0
            \end{array}
        \right.
    \end{equation}
    In other words, there is no vertex adjacent to both $e_2$ and $e_4$. 
    Therefore, the length of any path joining $e_2$ and $e_4$ must be at least 3. It follows that
    \begin{equation}\label{eq.ex19.>=3}
        \diam(\Gamma(\mathfrak s_{4,11})) \geq 3.
    \end{equation}
    Besides, we observe from Equations~\eqref{eq.ex19.41} and \eqref{eq.ex19.45} that 
    \begin{equation}
        \left\{ 
            \begin{array}{ll}
                [x_1e_1+x_3e_3+x_4e_4,e_3] = 0, 
                \\
                {}[x_1e_1+x_2e_2+x_3e_3+x_4e_4,x_2e_1-x_4e_3]  = 0
            \end{array}
        \right.
    \end{equation}
    for every $x_1, x_2, x_3, x_4 \in \mathbb F$. 
    Therefore, for two arbitrary vertices $u$ and $v$ in $\Gamma(\mathfrak s_{4,11})$, there are two vertices of the forms $ae_1+be_3$ and $ce_1+de_3$ such that $\{u, ae_1+be_3,ce_1+de_3,v\}$ is a path of $\Gamma(\mathfrak s_{4,11})$. That is,
    \begin{equation}\label{eq.ex19.<=3}
        \diam(\Gamma(\mathfrak s_{4,11})) \leq 3.
    \end{equation}
    Combining this with Equation~\eqref{eq.ex19.>=3}, we get
    \begin{equation}
        \diam(\Gamma(\mathfrak s_{4,11}))=3.
    \end{equation} 
\end{example}
To close this section, we remark that the speed of Process~\ref{algorithm} depends on the starting set of vertices. For instance, in Example~\ref{examplen411}, we may verify by the same method that $N^2(e_1)=\Gamma(\mathfrak s_{4,11})$. Therefore, if we choose $S=\{e_1\}$, we just need two iterations to obtain the desired equality $N^2(e_1)=\Gamma(\mathfrak s_{4,11})$. Thus, it gives a better upper bound: $\diam(\mathfrak s_{4,11}) \leq 4$ compared to Equation~\eqref{eq.<=6}. This results comes from a fact that $\dim(\ker(\ad_{e_1})) = 3 > \dim(\ker(\ad_{e_4}))=2$.
\section{Connectedness and Diameters of commuting graphs of Lie algebras having small-dimensional derived algebras} \label{sec:one-two-dimension}

In this section, we will determine the diameters of commuting graphs of finite-dimensional Lie algebras having either one- or two-dimensional derived algebras. For convenience, we will denote by $\mathrm{Lie}(n,k)$ the family of all Lie algebras of dimension $n$, which have derived algebras of dimension $k$. As usual, the derived Lie algebra of a Lie algebra $\mathcal L$ will be denoted by $\mathcal L^1$.  
We first show in the following proposition that most of the commuting graphs of Lie algebras having small-dimensional derived algebras have diameter 2.

\begin{proposition}\label{pro:m<n-2k}
    Let $\cal L$ be a Lie algebra of dimension $n$ such that $\dim Z(\mathcal L) = m$ and $\dim \mathcal L^1 = k$.
    \begin{itemize}
        \item[$(i)$] If $k<m$ then $\cal L$ is decomposable.
        \item[$(ii)$] If $m<n-2k$ then $\Gamma(\mathcal L)$ is connected, and $\diam(\Gamma(\mathcal L))=2$.
    \end{itemize}
\end{proposition}

\begin{proof}
    If $k<m$ then $Z(\mathcal L) \setminus \mathcal L^1 \neq \emptyset$.  
    This yields the existence of an element $x \in \mathcal L \setminus \mathcal L^1$ such that $[x,L]=0$. In the other words, $\mathcal L$ is decomposed into $\mathcal L' \oplus \left< x \right>$, as required.

    Now, suppose that $m<n-2k$. Let $u$ and $v$ be two arbitrary vertices in $\Gamma(\mathcal L)$. Then 
    \begin{equation}
        \dim \left(\ker(\ad_u)\right) \geq n-\dim \left(\mathrm{Im}(\ad_u)\right) \geq n-\dim \mathcal L^1 = n-k.
    \end{equation}
    Similarly, $\dim \left(\ker (\ad_v)\right) \geq n-k$. Therefore, 
    \begin{equation}
        \dim \left(\ker (\ad_u)  \cap \ker(\ad_v) \right) \geq n-2k>m.
    \end{equation}
    It follows that there exists a vertex $w \in \Gamma(\mathcal L)$ such that $w \in \ker(\ad_u) \cap \ker (\ad_v)$. In other words, $\{u, w, v\}$ is a path, which has distance $2$. Therefore,
    \begin{equation}
        \dist(u,v) \leq 2 \text{ for every } u, v \in \Gamma(\mathcal L).
    \end{equation}
    This proves the second part of the proposition.
\end{proof}

According to the above proposition, the majority of commuting graphs of Lie algebras with small-dimensional derived algebras are connected and have a diameter of two.   
For example, $\Gamma(\mathcal L)$ is connected, and $\diam(\Gamma(\mathcal L))=2$ if either $\dim \mathcal L^1 =1$ and $\dim Z(\mathcal L)<n-2$, or $\dim \mathcal L^1 =2$ and $\dim Z(\mathcal L)<n-4$.

\begin{lemma}\label{lem:centre<n-2}
    Let $\mathcal{L}$ be a noncommutative Lie algebra of finite dimension. Then $\dim Z(\mathcal L) \leq n-2.$
\end{lemma}

\begin{proof}
    By contradiction, suppose that $\dim Z(\mathcal{L})=\dim \mathcal L-1$. If so, there is a basis $e_1, \ldots, e_{n}$ of $\mathcal{L}$ such that $Z(\mathcal L)=\left< e_1, \ldots, e_{n-1}\right>$. It follows that $[e_i, e_j]=0$ for every $1 \leq i \leq j \leq n$. That is, $\mathcal L$ is commutative, a contradiction.
\end{proof}

\begin{lemma}\label{lemma:Lie(n,1)}
    Let $\mathcal L$ be a Lie algebra in the family $\mathrm{Lie}(n,1)$ such that $\dim Z(\mathcal{L})=n-2$. 
    Then $\Gamma(\mathcal L)$ is disconnected.
\end{lemma}

\begin{proof}
    It is direct from Proposition~\ref{Prop.disconnected}.
\end{proof}

Besides, it is proven in \cite[Theorem 1]{Schobel1993} that an indecomposable Lie algebra over the real field in the family $\mathrm{Lie}(n,1)$ is isomorphic to one of the followings: 
    \begin{align}
        \mathfrak n_{2,1}:=\left < e_1, e_2 \right> \colon &  [e_1,e_2]=e_1, \label{eq.n21}\\
        \mathfrak h_{2m+1}:=\left< e_1, e_2, e_{3}, \ldots, e_{2m+1}\right> \colon & [e_1,e_{1+m}]=[e_2,e_{2+m}]=\cdots = [e_m,e_{2m}] = e_{2m+1}. \label{eq.h3}
    \end{align}
The proof given in \cite[Theorem 1]{Schobel1993} is also available on every field. Therefore, any Lie algebra in the family $\mathrm{Lie}(n,1)$ must be isomorphic to one of the Lie algebras given in Equations~\eqref{eq.n21} and \eqref{eq.h3}.
Note that $Z(\mathfrak h_{2m+1}) = \left< e_{2m+1}\right>$. Therefore, by directly applying Proposition~\ref{pro:m<n-2k} and Lemma~\ref{lemma:Lie(n,1)}, we get the following theorem.

\begin{theorem}\label{theorem:Lie(n,1)}
    Let $\mathcal L$ be 
    an indecomposable Lie algebra in the family $\mathrm{Lie}(n,1)$. Then the commuting graph of $\mathcal L$ is connected if and only if $\mathcal L$ is isomorphic to $\mathfrak h_{2m+1}$ for some $m\geq 2$. Moreover,  the following equality holds for every $m \geq 2$:
    \begin{equation}
        \diam\left(\Gamma(\mathfrak h_{2m+1}) \right)=2.
    \end{equation}
    
\end{theorem}

\begin{remark} Note that the centers of $\mathfrak n_{2,1}$ and $\mathfrak h_3$ defined in Equations~\eqref{eq.n21} and \eqref{eq.h3} are:
\begin{equation}
    Z(\mathfrak n_{2,1}) = \{0\}, \quad \text{ and } \quad Z(\mathfrak h_3)=\left< e_3\right>.
\end{equation}
\begin{itemize}
    \item Let $a_1e_1+a_2e_2$ and $b_1e_1+b_2e_2$ be two vertices of $\Gamma(\mathfrak n_{2,1})$. Then 
    \begin{equation}\label{eq41}
        [a_1e_1+a_2e_2,b_1e_1+b_2e_2] = 0 \iff a_1b_2-a_2b_1=0.
    \end{equation}
    Let $S_{1,a}:=\{\lambda(e_1+ ae_2) \mid 0 \neq \lambda\in \mathbb F\}$ and $S_2:=\{\lambda e_2 \mid 0 \neq \lambda\in \mathbb F\}$. Then we may conclude from Equation~\eqref{eq41} that each subgraph of $\Gamma(\mathfrak n_{2,1})$ induced by $S_{1,a}$ (for some $a \in \mathbb F$) or by $S_2$ is a connected component of $\Gamma(\mathfrak n_{2,1})$. In particular, if $\mathbb F$ is finite, then the number of connected components of $\Gamma(\mathfrak n_{2,1})$ is $(|\mathbb F|+1)$ and each connected component is a complete graph of $|\mathbb F|-1$ vertices.

    \item Similarly, for two distinct vertices $a_1e_1+a_2e_2+a_3e_3$ and $b_1e_1+b_2e_2+b_3e_3$ of $\Gamma(\mathfrak h_3)$, we can check that
    \begin{equation}
        [a_1e_1+a_2e_2+a_3e_3,b_1e_1+b_2e_2+b_3e_3] = 0 \iff a_1b_2-a_2b_1=0.
    \end{equation}
    Now, let $S_{1,a}:=\{\lambda(e_1+ae_2+be_3) \mid 0 \neq \lambda \in \mathbb F, b \in \mathbb F\}$, and $S_{2}:=\{\lambda(e_2+be_3) \mid 0 \neq \lambda \in \mathbb F, b \in \mathbb F\}$ for some $a \in \mathbb F$. Then the subgraph induced by one of the sets $S_{1,a}$ and $S_2$, is a connected component of $\Gamma(\mathfrak h_3)$. In particular, if $\mathbb F$ is finite, then $\Gamma(\mathfrak h_3)$ has $(|\mathbb F|+1)$ connected components, and each connected component is a complete graph of $|\mathbb F|(|\mathbb F|-1)$ vertices. 
\end{itemize}
    
\end{remark}

Next, we study the case where $\mathcal L$ belongs to the family $\mathrm{Lie}(n,2)$. If $\dim Z(\mathcal{L}) < n-4$ then $\Gamma(\mathcal L)$ is connected and $\diam (\Gamma(\mathcal L)) = 2$. By Lemma~\ref{lem:centre<n-2}, $\dim Z(\mathcal{L}) \leq  n-2$ unless $\mathcal{L}$ is commutative. If $\dim Z(\mathcal{L})=n-2$ then there is a basic $\{e_1, \ldots, e_n \}$ 
of $\mathcal{L}$ such that $[e_{i},e_j] = 0$ for every $i<j$ with $(i,j) \neq  (n-1, n).$ It follows that $\dim \mathcal L^1 \leq 1$, a contradiction. 
This proves the following lemma.

\begin{lemma}\label{lem:dim<n-3}
    Let $\mathcal{L}$ be a noncommutative Lie algebra in the family $\mathrm{Lie}(n,2)$. Then $\dim Z(\mathcal L) \leq n-3$.
\end{lemma}
Now, applying directly Proposition~\ref{Prop.disconnected}, we get the following result.

\begin{lemma}
    Let $n$ be an integer such that $n\geq 3$ and let $\cal L$ be a Lie algebra in the family $\mathrm{Lie}(n,2)$ such that $\dim Z(\mathcal L)=n-3$. Then $\Gamma(\mathcal L)$ is disconnected.
\end{lemma}

Next, we will classify all the Lie algebras in the family $\mathrm{Lie}(n,2)$ having centers of codimension 4 and determine the diameters of their commuting graphs.

\begin{theorem}\label{theorem:Lie(n,2)-classification}
    Let $n$ be an integer such that $n\geq 4$ and let $\cal L$ be a Lie algebra in the family $\mathrm{Lie}(n,2)$ such that $\dim Z(\mathcal L)=n-4$. Then $\mathcal L = \mathcal M \oplus \mathbb{F}^{n-k}$ for some integer $4 \leq k\leq 6$, and some Lie algebra, $\mathcal M$, belonging to the family $\mathrm{Lie}(k,2)$. Moreover, $\mathcal M$  is isomorphic to one of the following Lie algebras:
    \begin{align}
        \mathcal M_{1,k} = \left< e_1, e_2, \ldots, e_k \right> \colon & [e_1, e_2] = \sum_{i=1}^k a_i e_i, \, [e_3, e_4] = \sum_{i=1}^k b_i e_i \label{Lie.M1}\\
        \mathcal M_{2,k} = \left< e_1, e_2, \ldots, e_k \right> \colon & [e_1, e_2] = \sum_{i=1}^k a_i e_i, \, [e_1, e_3] = \sum_{i=1}^k b_i e_i, \, [e_3,e_4] = \sum_{i=1}^k (a_i+\alpha b_i) e_i \label{Lie.M2}\\
        \mathcal M_{3,k} = \left< e_1, e_2, \ldots, e_k \right> \colon &  [e_1, e_2] = [e_2,e_4]=\sum_{i=1}^k a_i e_i, \, [e_1, e_3] = \sum_{i=1}^k b_i e_i, \,  [e_3,e_4] = \sum_{i=1}^k (\beta a_i+\gamma b_i) e_i \label{Lie.M3}
    \end{align}
    for some $a_i, b_i, \alpha, \beta,\gamma \in \mathbb F$ such that $(a_1,\ldots,a_k)$ and $(b_1,\ldots,b_k)$ are linearly independent over $\mathbb F$ and that $x^2+\gamma x - \beta$ has no root in $\mathbb F$. Furthermore, the commuting graphs of $\mathcal M_{1,k}$ and $\mathcal M_{2,k}$ are connected and both have diameter $3$, while the commuting graph of $\mathcal M_{3,k}$ is disconnected. 
\end{theorem}

\begin{proof}
    Since $\dim Z(\mathcal L)=n-4$, there is a basic $\{e_1, \ldots, e_n\}$ of $\mathcal L$ such that 
    \begin{equation*}
        Z(\mathcal L) = \left< e_{5}, \ldots, e_n \right>, \text{ and } \mathcal{L}^1 = \left< [e_i, e_j]: 1\leq i< j\leq 4\right>. 
    \end{equation*}
Because $\dim \mathcal L^1=2$, we have only two possibilities for the spanning set of $\mathcal{L}^1$ as follows:
\\[5pt]
\noindent\textbf{Case 1.} There is no $i, j, k$ such that $\mathcal{L}^1 = \left< [e_i, e_j], [e_i, e_k]\right>$. In this case, by changing the indices if necessary, we may assume that $\mathcal L^1 = \left< [e_1, e_2], [e_3, e_4]\right>$ and
\begin{equation}
    [e_1, e_3] = [e_1, e_4] = [e_2, e_3]=[e_2,e_4]=0.
\end{equation}
It follows that 
\begin{align}
    & [c_1e_1+c_2e_2,d_3e_3+d_4e_4]= 0 \quad \text{for every } c_1, c_2, d_3, d_4 \in \mathbb F.
\end{align}
Therefore, for arbitrary two vertices $u=\sum_{i=1}^n c_ie_i$ and $v=\sum_{i=1}^n d_ie_i$,
\begin{itemize}
    \item $\{u,e_1+e_2,v\}$ is a path of $\Gamma(\mathcal L)$ provided that $(c_1,c_2) = (0,0)$ and $(d_1,d_2)=(0,0)$,

    \item $\{u,d_1e_1+d_2e_2,v\}$ is a path of $\Gamma(\mathcal L)$ provided that $(c_1,c_2) = (0,0)$ and $(d_1,d_2) \neq (0,0)$,

    \item $\{u,c_1e_1+c_2e_2,e_3+e_4,v\}$ is a path of $\Gamma(\mathcal L)$ provided that $(c_1,c_2) \neq  (0,0)$ and $(d_3,d_4) = (0,0)$,
    
    \item $\{u,c_1e_1+c_2e_2,d_3e_3+d_4e_4,v\}$ is a path of $\Gamma(\mathcal L)$ provided that $(c_1,c_2) \neq  (0,0)$ and $(d_3,d_4) \neq (0,0)$.
\end{itemize}
%
%
Therefore, any pair of vertices of $\Gamma(\mathcal L)$ is joined by a path of length at most $3$. Moreover, it can be seen that there is no vertex in  $\Gamma(\mathcal L)$ adjacency to both $e_1+e_3$ and $e_2+e_4$. We thus conclude that $\Gamma(\mathcal L)$ is connected and $\diam(\Gamma(\mathcal L))=3$ in this case. 

Furthermore, if we denote by $\cal H$ the subalgebra of $\mathcal L$ generated by $\{e_1, e_2, e_3, e_4, [e_1,e_2], [e_3,e_4]\}$, then $\mathcal L$ is isomorphic to $\mathcal H \oplus \mathbb F^{n-\dim \mathcal H}$. This Lie algebra $\mathcal H$ is clearly isomorphic to $\mathcal M_{1,k}$ where $k=\dim \mathcal H$. Therefore, in this case, $\mathcal L$ is isomorphic to $\mathcal M_{1,k} \oplus \mathbb F^{n-k}$, and $\dim (\Gamma(\mathcal L)) = \dim (\Gamma(\mathcal M_{1,k}) = 3.$
\\[5pt]
\noindent\textbf{Case 2.} There is $i, j, k$ such that $\mathcal{L}^1 = \left< [e_i, e_j], [e_i, e_k]\right>$. Without loss of generality, assume that $\mathcal{L}^1 = \left< [e_1, e_2], [e_1,e_3]\right>$. If so, there are $c_1, \ldots, c_6 \in \mathbb F$ such that
\begin{align}
    & [e_1,e_4]=c_1[e_1, e_2]+c_2[e_1,e_3], \\
    & [e_2,e_3]=c_3[e_1, e_2]+c_4[e_1,e_3], \\ 
    & [e_2,e_4]=c_5[e_1, e_2]+c_6[e_1,e_3].
\end{align}
Let $e_2'=e_2-c_4e_1$, $e_3'=e_3+c_3 e_1$ and $e_4'=e_4-(c_1c_4-c_5)e_1-c_1 e_2'-c_2 e_3'$. Then
\begin{align}
    [e_2',e_3'] = [e_1,e_4']=0, \text{ and }
    [e_2',e_4'] =  (c_6-c_2c_4) [e_1,e_3'].
\end{align}
Therefore, without loss of generality, we may assume that 
\begin{equation}\label{eq.55}
    [e_1,e_4]=[e_2,e_3]=0,  [e_2,e_4]=\alpha [e_1, e_3], \text{ and } [e_3,e_4]=\beta [e_1,e_2]+\gamma [e_1, e_3].
\end{equation}
for some $\alpha, \beta, \gamma$ in $\mathbb F$. We split this case into three subcases as follows:
\begin{description}
    \item[Case 2.1.] $\alpha=\beta=0$. If so, $\gamma \neq 0$ holds because $e_4 \notin Z(\mathcal L)$. By letting $e_1'=\gamma e_1+e_4$, we easily see that 
    \begin{equation*}
        [e_1',e_3]=[e_1',e_4]=0.
    \end{equation*}
    Therefore, $\mathcal L$ is isomorphic to $\mathcal M_{1,k} \oplus \mathbb F^{n-k}$. As a consequence, its commuting graph, $\Gamma(\mathcal{L})$, is connected and has diameter $3$. 
    
    \item[Case 2.2.] $\alpha = 0$, and $\beta \neq 0$. Let $u=\sum_{i=1}^n c_ie_i$ and $v=\sum_{i=1}^n d_i e_i$ be two arbitrary vertices in $\Gamma(\mathcal{L})$. 
    \begin{itemize}
        \item \textbf{Case 2.2.1.} $c_1 \neq 0$. If so, by letting 
    \begin{equation}
        u_1=(c_1\gamma)e_1 + (c_2\gamma-c_3\beta)e_2 +  c_1e_4,
    \end{equation}
    we observe that $u_1$ is an adjacent vertex of $u$ in $\Gamma(\mathcal L)$.
    Moreover, by letting $u_2=x_1e_1+x_2e_2+x_3e_3+x_4e_4 \in \mathcal L$, we get 
    \begin{equation}
        \left\{
            \begin{aligned}
                 [u_1,u_2] = 0 \\
                 {}[v,u_2] = 0
            \end{aligned}
        \right. 
        \Longleftrightarrow 
        \left\{
            \begin{aligned}
                a_1\gamma x_2 - (a_2\gamma-a_3\beta)x_1 - a_1\beta x_3 & = 0 \\
                a_1x_2-a_2x_1 + \beta (a_3x_4-a_4x_3) & = 0 \\
                a_1x_3-a_3x_1+\gamma(a_3x_4-a_4x_3) & =0
            \end{aligned}
        \right.
    \end{equation}
    The above homogeneous system must has a non-trivial solution. Therefore, there is a vertex $u_2$ in $\Gamma(\mathcal L)$ such that $\{u,u_1,u_2,v\}$ is a path. 

    \item \textbf{Case 2.2.2.} $c_1=d_1=0$. If so, we observe from the assumption $[e_2,e_3]=[e_2,e_4]=0$ in Equation~\eqref{eq.55} that $\{u, e_2,v\}$ is a path in $\Gamma(\mathcal L)$. 
    
    \item \textbf{Case 2.2.3.} $c_1=0$ and $d_1 \neq 0$. Similar to \textbf{Case 2.2.2} above, we observe that any pair of vertices, $u$ and $v$, is joined by a path of length $3$.
    \end{itemize}
    Therefore, in this \textbf{Case 2.2}, we observe that $\Gamma(\mathcal L)$ is connected and 
    \begin{equation} \label{eq.<=3}
        \diam(\Gamma(\mathcal L)) \leq 3.
    \end{equation}
    Besides, according to the structure constants given in Equation~\eqref{eq.55}, we observe that
    \begin{equation}
        \left\{
            \begin{aligned}
               & \left[ e_1,\sum_{i=1}^n c_i e_i \right] = 0 \\
               & \left[ e_3,\sum_{i=1}^n c_i e_i \right] = 0
            \end{aligned}
        \right.
        \iff 
        \left\{
            \begin{aligned}
                c_2 = c_3  = 0 \\
                c_1 = c_4 = 0
            \end{aligned}
        \right.
    \end{equation}
    Therefore, there is no vertex $w$ in $\Gamma(\mathcal L)$ such that $w$ is adjacent to both the vertices $e_1$ and $e_3$. In combining with Equation~\eqref{eq.<=3}, we get $\diam(\Gamma(\mathcal L)) =3$.

    Now, by basis change $\frac{1}{\beta}e_4 \rightarrow e_4$, we easily see that $\mathcal L$ is isomorphic to $\mathcal M_{2,k}$ where $k$ equals to the dimension of the vector space generated by  $\{e_1, e_2, e_3, e_4, [e_1,e_2], [e_1,e_3]\}$.

    \item[Case 2.3.] $\alpha \neq 0$. Since $[e_2, \frac{1}{\alpha} e_4]=\frac{1}{\alpha}[e_2, e_4] = [e_1,e_3]$, we may assume (after a change of basis if necessary) that $\alpha=1$. 
    \begin{itemize}
        \item \textbf{Case 2.3.1.} The quadratic $x^2+\gamma x-\beta$ has no root in $\mathbb F$.
        Let $b_1,\, b_2$ be two elements in $\mathbb F$ such that $(b_1,b_2) \neq (0,0)$ and let $V_{b_1,b_2}$ be the set of all vertices in $\Gamma(\mathcal L)$ of the form $\lambda b_1 e_1 + \lambda b_2 e_2 + \delta b_2e_3-\delta b_1e_4 +\sum_{i=5} b_i e_i$ where $(0,0) \neq (\lambda,\delta) \in \mathbb F^2$ and $b_5, \ldots, b_n$ in $\mathbb F$. That is, 
        \begin{equation}
            V_{b_1,b_2}=\left\{\lambda b_1 e_1 + \lambda b_2 e_2 + \delta b_2e_3-\delta b_1e_4 +\sum_{i=5} b_i e_i: (\lambda,\delta) \neq (0,0) \right\} \subseteq V_{\cal L}.
        \end{equation}
        We will prove that the subgraph of $\Gamma(\cal L)$ induced by $V_{b_1,b_2}$, denoted by $\Gamma(V_{b_1,b_2})$,  is a connected component of $\Gamma(\mathcal L)$. 
        Indeed, let $u=\sum_{i=1}^n a_i e_i$ and $v=\lambda b_1 e_1 + \lambda b_2 e_2 + \delta b_2e_3-\delta b_1e_4 +\sum_{i=5}^n b_i e_i$ be two arbitrary vertices in $\Gamma(\mathcal L)$ and $V_{b_1, b_2}$, respectively. Then 
        \begin{equation}
            [u,v]=0 \iff \left\{
                \begin{aligned}
                    & \lambda (a_1b_2-a_2b_1)-\beta\delta (a_3b_1+a_4b_2) =0 \\
                    & \delta a_1b_2-\lambda a_3b_1-\delta a_2b_1-\lambda a_4b_2-\gamma\delta (a_3b_1+a_4b_2) =0
                \end{aligned}
            \right.
        \end{equation}
        It follows that
        \begin{equation}\label{eq:system}
            [u,v]=0 \iff \left\{
                \begin{aligned}
                    & \lambda(a_1b_2-a_2b_1)-\beta\delta(a_3b_1+a_4b_2) =0 \\
                    & \delta(a_1b_2-a_2b_1)-(\gamma\delta+\lambda)(a_3b_1+a_4b_2) =0
                \end{aligned}
            \right.
        \end{equation}
        It is elementary to check that 
        \begin{equation}\label{eq:condi}
            -\lambda(\gamma\delta+\lambda) + \beta\delta^2 = -\lambda^2-\gamma(\lambda\delta)+\beta\delta^2.
        \end{equation}
        Since $x^2+\gamma x-\beta$ has no root in $\mathbb F$, Equation~\eqref{eq:condi} implies that $-\lambda(\gamma\delta+\lambda) + \beta\delta^2 \neq 0$. Thus, we observe from Equation~\eqref{eq:system} that
        \begin{equation}
            [u,v]=0 \iff a_1b_2 - a_2b_1=a_3b_1+a_4b_2=0.
        \end{equation}
        Therefore, $u$ is adjacent to a vertex in $\Gamma(V_{b_1,b_2})$ if and only if $u \in V_{b_1,b_2}$. In other words, the induced subgraph of $\Gamma(\mathcal L)$ generated by $V_{b_1,b_2}$ is a connected component of $\Gamma(\mathcal L)$. It follows that $\Gamma(\mathcal L)$ is disconnected. Moreover, from the above proof, we may see that each connected component is a complete graph. In the case where $\mathbb F$ is a finite field, then the number of connected components is $(|\mathbb F|+1)$ and each connected component is a complete graph of $(|\mathbb F|^2-1)|\mathbb F|^{n-4}$ vertices.
        
        Note that this Lie algebra is isomorphic to $\mathcal M_{3,k} \oplus \mathbb F^{n-k}$, where $k$ equals to the dimension of the vector space generated by $\{e_1, e_2, e_3, e_4, [e_1, e_2], [e_1, e_3]\}$.

        \item \textbf{Case 2.3.2.} The quadratic $x^2+\gamma x-\beta$ has a root $\zeta$ in $\mathbb F$. In this case, $\beta=\zeta(\gamma+\zeta)$. Therefore, 
        \begin{align}
            & [e_3+\zeta e_2,e_4+(\gamma+\zeta)e_1] \notag\\ 
            = & [e_3,e_4]+\zeta[e_2,e_4]-(\gamma+\zeta)[e_1,e_3]-\zeta(\gamma+\zeta)[e_1,e_2] \notag\\
            = &  \beta[e_1,e_2] + \gamma[e_1,e_3] + \zeta[e_1,e_3]-(\gamma+\zeta)[e_1,e_3]-\zeta(\gamma+\zeta)[e_1,e_2] =0.
        \end{align}
        Moreover, we may check directly the following equality:
        \begin{equation}
            [e_2,e_4+(\gamma+\zeta)e_1] = [e_1, e_3+\zeta e_2]  - (\gamma + 2\zeta) [e_1, e_2].
        \end{equation}
        Thus, this Lie algebra is isomorphic to $\mathcal M_{2,k}\oplus\mathbb F^{n-k}$. Therefore, the commuting graph of Lie algebras in this case is connected and has diameter $3$.
    \end{itemize}
        


\end{description} 
\end{proof}

\begin{remark}  
In the case where the characteristic of $\mathbb F$ is different from $2$, the polynomial $x^2+\gamma x - \beta$ has no root in $\mathbb F$ if and only if $\gamma^2+4\beta$ is not a square in $\mathbb F$. Otherwise, $x^2+\gamma x-\beta$ does not have root in a field $\mathbb F$ of characteristic $2$ if and only if $\gamma\neq 0$ and $\dfrac{-\beta}{\gamma} + \left(\dfrac{-\beta}{\gamma}\right)^2 + \cdots + \left(\dfrac{-\beta}{\gamma}\right)^{2^{n-1}}=1,$ where $n = [\mathbb F:\mathbb F_2]$ is the degree of extension of the field of $\mathbb F$ over $\mathbb F_2$, the field of 2 elements \cite[Corollary 3.79]{Lidl1987}. 
\end{remark}

\section{Diameters of Commuting graphs of Lie algebras of small dimension}\label{sec:<=5}

In this section, we will use results in previous sections, especially Process in Section~\ref{sec:algorithm} to study the diameters of the commuting graphs of all real solvable Lie algebras of dimension up to 4. These Lie algebras are classified in \cite[Chapters 16, 17, 18]{Snob}. Note that these results are applied similarly to the complex case. Recall that the commuting graph of any Lie algebra of dimension $\leq 3$ is disconnected (see Corollary \ref{cor.dim<=3}). 
\begin{theorem}
    Let $\cal L$ be a solvable Lie algebra of dimension 4 over $\mathbb R$. Then $\Gamma(\mathcal L)$ is connected if and only if $\mathcal L$ is isomorphic to one of the following Lie algebras: $\mathfrak n_{2,1} \oplus \mathfrak n_{2,1}$ and $\mathfrak s_{4,11}$, where $\mathfrak n_{2,1}$ is defined in Equation~\eqref{eq.n21}, and $\mathfrak s_{4,11}$ is defined as follows:
    \begin{equation}\label{eq.s411}
        \mathfrak s_{4,11} = \left< e_1, e_2, e_3, e_4 \right> \colon [e_2,e_3]=e_1, \, [e_4,e_1] = e_1, \, [e_4,e_2]=e_2.
    \end{equation}
    Moreover, $\diam\left(\Gamma(\mathfrak n_{2,1} \oplus \mathfrak n_{2,1}) \right) = \diam \left(\Gamma(\mathfrak s_{4,11})\right) = 3.$
\end{theorem}
\begin{proof} 
    According to Proposition~\ref{Prop.disconnected}, we observe that $\Gamma(\mathcal L)$ is disconnected if $\dim(Z(\mathcal L)) \geq 1$. Therefore, in the case where $\mathcal L$ is decomposable, the commuting graph of $\mathcal L$ is connected only if $\mathcal{L}$ is isomorphic to $\mathcal L_1 \oplus \mathcal L_2$ where $\mathcal L_1$ and $\mathcal L_2$ are two noncommutative subalgebras of $\mathcal L$ satisfying $\dim \mathcal L_1 = \dim \mathcal L_2 =2$. In this case, $\Gamma(\mathcal L)$ is connected, and $\diam(\Gamma(\mathcal L))=3$ (see Remark~\ref{remark}). Note that any noncommutative Lie algebra of dimension 2 is isomorphic to the Lie algebra $\mathfrak n_{2,1}$ defined in Equation~\eqref{eq.n21}. Therefore, in the case where $\mathcal L$ is decomposable, $\mathcal L$ is connected if and only if $\mathcal L$ is isomorphic to $\mathfrak n_{2,1} \oplus \mathfrak n_{2,1}$.

    Also, if $\mathcal L$ is indecomposable, then $\mathcal L$ must be isomorphic to one of the nonisomorphic $13$ families: $\mathfrak n_{4,1}$, $\mathfrak s_{4,1}, \ldots, \mathfrak s_{4,12}$ (see \cite[Chapter 17]{Snob}). Note that the families $\mathfrak n_{4,1}, \, \mathfrak s_{4,1}$, $\mathfrak s_{4,6}$, and $\mathfrak s_{4,7}$ have a non-trivial center. Therefore, the commuting graphs of $\mathfrak n_{4,1}, \, \mathfrak s_{4,1}$, $\mathfrak s_{4,6}$, and $\mathfrak s_{4,7}$ are disconnected. These families of Lie algebras and their centers are given in Table~\ref{table:1.2}.

    \begin{table}[!ht]
\centering
{\small
\caption{4-dimensional solvable Lie algebras having non-trivial centers.}
\label{table:1.2}
 \begin{tabular}{c l c c c} 
 \hline
 $\mathcal L$ & Non-trivial Lie brackets & $Z(\mathcal L)$ & $\dim \mathcal L-\dim Z(\mathcal L)$ & Connectedness\\ 
 \hline
$\mathfrak n_{4,1}$ & $[e_2,e_4]=e_1, [e_3,e_4]=e_2$ &  $\left< e_1\right>$ & $3$ & Disconnected
\\
$\mathfrak s_{4,1}$ & $[e_4,e_2]=e_1, [e_4,e_3]=e_3$ &  $\left< e_1 \right>$ & $3$ & Disconnected
 \\
 $\mathfrak s_{4,6}$ & $[e_2,e_3]=e_1, [e_4,e_2]=e_2, [e_4,e_3]=-e_3$ & $\left< e_1 \right>$ & $3$ & Disconnected
 \\
 $\mathfrak s_{4,7}$ & $[e_2,e_3]=e_1, [e_4,e_2]=-e_3, [e_4,e_3]=e_2$ & $\left< e_1 \right>$ & $3$ & Disconnected
\\
 \hline
 \end{tabular}  }
\end{table}
    
    Now, we will show that each of the families $\mathfrak s_{4,2},\, \mathfrak s_{4,3},\, \mathfrak s_{4,4},\, \mathfrak s_{4,5},\, \mathfrak s_{4,8},\, \mathfrak s_{4,9},\, \mathfrak s_{4,10}$ is disconnected by using Corollary~\ref{coro1}. These families are defined in Table~\ref{table:1.1}. 
    Let us first consider the algebra $\mathfrak s_{4,2}$. In this algebra, we easily see that 
    \begin{equation}
        \begin{array}{ll}
            [x_1e_1+x_2e_2+x_3e_3+x_4e_4,e_4] = 0 & \iff -x_1e_1-x_2(e_1+e_2)-x_3(e_2+e_3)=0 \\
            & \iff x_1=x_2=x_3=0.
        \end{array} 
    \end{equation}
    Therefore, $\ker (\ad_{e_4}) = \left< e_4\right>$. It follows that $\dim \left(\ker (\ad_{e_4})\right)=1= \dim Z(\mathfrak s_{4,2})+1$. Applying Corollary~\ref{coro1}, the commuting graph of $\mathfrak s_{4,2}$ is disconnected. Using the same method, we may show that, for a Lie algebra $\mathcal L$ in $\mathfrak s_{4,2},\, \mathfrak s_{4,3},\, \mathfrak s_{4,4},\, \mathfrak s_{4,5},\, \mathfrak s_{4,8}, \, \mathfrak s_{4,9},\, \mathfrak s_{4,10}$, there is a nonzero vector $x \in \mathcal L$ or a vertex $x \in V_{\mathcal L}$ such that $\dim \left(\ker (\ad_x)\right) = \dim Z(\mathcal L)+1$ (see Table~\ref{table:1.1}), and thus the commuting graph of each Lie algebra in $\mathfrak s_{4,2}$, $\mathfrak s_{4,3}$, $\mathfrak s_{4,4}$, $\mathfrak s_{4,5}$, $\mathfrak s_{4,8},$ $\mathfrak s_{4,9},$ $\mathfrak s_{4,10}$, is disconnected.

\begin{table}[!ht]
\centering
{\small
\caption{Solvable Lie algebras, $\mathcal L$, of dimension $4$ satisfying $\dim \left(\ker (\ad_x)\right)=\dim Z(\mathcal L)+1$ for some $x\in V_{\mathcal L}$. Here, all Lie algebras are assumed to be spanned by $\{e_1, e_2, e_3, e_4\}$.}
\label{table:1.1}
 \begin{tabular}{c l l l l l} 
 \hline
$\mathcal L$ & Non-trivial Lie brackets & $x$ & $\ker (\ad_x)$ & Connectedness\\ 
 \hline
$\mathfrak s_{4,2}$ & $
\left\{\begin{aligned}
    & [e_4,e_1]=e_1, [e_4,e_2]=e_1+e_2, \\
    & [e_4,e_3]=e_2+e_3
\end{aligned}\right.
$ & $e_4$ & $\left< e_4 \right>$ & Disconnected
 \\
 $\mathfrak s_{4,3}$ & $\left\{\begin{aligned}
 & [e_4,e_1]=e_1, [e_4,e_2]=\alpha e_2, [e_4,e_3]=\beta e_3 \\
 & \left(0<|\beta|\leq |\alpha| \leq 1, (\alpha,\beta) \neq (-1,-1) \right)
 \end{aligned}\right.$ &  $e_4$ & $\left< e_4 \right>$ & Disconnected
 \\
 $\mathfrak s_{4,4}$ & $\left\{\begin{aligned} & [e_4,e_1]=e_1, [e_4,e_2]=e_1+e_2, \\ & [e_4,e_3]=\alpha e_3 \quad (\alpha \neq 0)\end{aligned}\right.$ &  $e_4$ & $\left< e_4 \right>$ & Disconnected
 \\
 $\mathfrak s_{4,5}$ & $\left\{\begin{aligned} & [e_4,e_1]=\alpha e_1, [e_4,e_2]=\beta e_2-e_3, \\ & [e_4,e_3]=e_2+\beta e_3 \quad (\alpha > 0)\end{aligned}\right.$ & $e_4$ & $\left< e_4 \right>$ & Disconnected
 \\
 $\mathfrak s_{4,8}$ & $\left\{\begin{aligned} & [e_2,e_3]=e_1, [e_4,e_1]=(1+\alpha)e_1, [e_4,e_2]=e_2, \\ & [e_4,e_3]=\alpha e_3  \quad (-1< \alpha \leq 1)\end{aligned}\right.$& $e_4$ & $\left< e_4 \right>$ & Disconnected
 \\
$\mathfrak s_{4,9}$ & $\left\{\begin{aligned} & [e_2,e_3]=e_1, [e_4,e_1]=2\alpha e_1, [e_4,e_2]=\alpha e_2-e_3,\\ & [e_4,e_3]=e_2+\alpha e_3 \quad (\alpha >0)\end{aligned}\right.$ & $e_4$ & $\left< e_4 \right>$ & Disconnected
 \\
 $\mathfrak s_{4,10}$ &  $\left\{\begin{aligned} & [e_2,e_3]=e_1, [e_4,e_1]=2e_1, \\ & [e_4,e_2]=e_2, [e_4,e_3]=e_2+e_3\end{aligned}\right.$ & $e_4$ & $\left< e_4 \right>$ & Disconnected
 \\
 \hline
 \end{tabular}  }
\end{table}
    
    For the family $\mathfrak s_{4,11}$, a calculation similar to Example~\ref{examplen411} gives us the series of closed neighborhoods of $e_4$ in $\Gamma(\mathfrak s_{4,11})$ as presented in Table~\ref{table:1}. This series satisfies the equality $N^3(e_4)=\Gamma(\mathfrak s_{4,11})$. It follows that $\Gamma(\mathfrak s_{4,11})$ is connected and $\diam(\Gamma(\mathfrak s_{4,11})) \leq 6$ (see Theorem~\ref{theorem:diam<d}). Using the same argument as in Example~\ref{examplen411}, the equations~\eqref{eq.ex19.>=3} and \eqref{eq.ex19.<=3} both hold for the case where $\mathbb F=\mathbb R$ also. Therefore, $\diam(\Gamma(\mathfrak s_{4,11}))=3$. 
    
    Finally, for the family $\mathfrak s_{4,12}$, one can check that
    \begin{equation}
        [e_1,x_1e_1+x_2e_2+x_3e_3+x_4e_4] = 0 \iff -x_3e_1+x_4e_2=0 \iff x_3=x_4=0,
    \end{equation}
    and 
    \begin{equation}
        [a_1e_1+a_2e_2,x_1e_1+x_2e_2+x_3e_3+x_4e_4] = 0 \iff 
        \left\{ 
            \begin{array}{ll}
                a_1x_3+a_2x_4=0\\
                a_1x_4-a_2x_3=0
            \end{array}
        \right. \iff x_3=x_4=0,
    \end{equation}
    for every $(0,0) \neq (a_1,a_2) \in \mathbb R^2$.
    Therefore, 
    \begin{align}
        & N^1(e_1) = \left< e_1, e_2 \right> \setminus \{0\},\\
        & N^2(e_1) = \left<e_1,e_2\right> \setminus \{0\} = N^1(e_1).
    \end{align}
    By Theorem~\ref{theorem:diam<d}, the commuting graph of $\mathfrak s_{4,12}$ is disconneted. This completes the proof.
\begin{table}[!ht]
\centering
{\small
\caption{A series of closed neighborhoods of $\mathfrak s_{4,11}$ and $\mathfrak s_{4,12}$. Here, we denote by $\{e_1, e_2, e_3, e_4\}$ a basis for those Lie algebras, and by $U^*:=\{x\in U\mid x \neq \mathbf 0_U\}$ the set of non-zero elements of a vector space $U$.}
\label{table:1}
 \begin{tabular}{c l l l l l} 
 \hline
$\cal L$ & Non-trivial Lie brackets & $S$ & $N^1(S)$ & $N^2(S)$ & $N^3(S)$\\ 
 \hline
$\mathfrak s_{4,11}$ &  $\left\{\begin{aligned} & [e_2,e_3]=e_1, [e_4,e_1]=e_1, \\
& [e_4,e_2]=e_2 \end{aligned}\right.$ & $e_4$ & $\left<e_3,e_4\right>^*$ &  $\left<e_1, e_3,e_4 \right>^*$ & $\left< e_1,e_2,e_3,e_4\right>^* = V_{\mathfrak s_{4,11}}$\\
\hline
$\mathfrak s_{4,12}$ & $\left\{\begin{aligned} & [e_3,e_1]=e_1, [e_3,e_2]=e_2 \\ & [e_4,e_1]=-e_2, [e_4,e_2]=e_1 \end{aligned}\right.$ & $e_1$ & $\left<e_1,e_2\right>^*$ & $\left<e_1,e_2\right>^* \subsetneq V_{\mathfrak s_{4,12}}$
\\
 \hline
 \end{tabular}  }
\end{table}

\end{proof}

\section{Conclusion}
This study has explored several interesting properties related to the connectedness and diameter of the commuting graph of a general Lie algebra over any field. Because the diameter of the commuting graph of a decomposable Lie algebra can be easily computed via its subalgebras, it is sufficient to investigate only indecomposable Lie algebras. For indecomposable Lie algebras, we obtained an equivalent condition for these Lie algebras to have connected commuting graphs and proposed a process to determine an upper bound for their diameters. To demonstrate the theory as well as the process, we have determined the connectedness of: (1) the families of all Lie algebras having one- or two-dimensional derived algebras, and (2) real solvable Lie algebras of dimension up to 4. 
The diameters of these Lie algebras were also computed. Therefore, it is worth extending these results to all Lie algebras and using their commuting graphs to obtain a new classification of Lie algebras. Besides, if $\mathbb F$ is finite then the commuting graph of any finite-dimensional Lie algebra over $\mathbb F$ is finite which thus implies that the process \ref{algorithm} terminates for this case. To close, we propose a conjecture that this result is still valid if $\mathbb F$ is either the complex or the real fields. 

\section*{Acknowledgements}
This research is supported by the project B2024-SPD-07. 

\end{document}